\documentclass[preprint,12pt]{elsarticle}

\usepackage{amssymb, bm}
\usepackage{amsthm, amsmath}

\usepackage{graphicx,color}
\usepackage{caption}
\usepackage{subcaption}
\usepackage{multirow}
\usepackage{mathtools,xparse}

\usepackage{listings}

\newtheorem{lemma}{Lemma}
\newtheorem{thm}{Theorem}

\makeatletter
\newenvironment{smallermatrix}[1][c]
{\null\,\vcenter\bgroup
  \Let@\restore@math@cr\default@tag
  \baselineskip0pt \lineskip0.4pt \lineskiplimit0pt
  \ialign\bgroup\if#1l\else\hfil\fi$\m@th\scriptstyle##$\if#1r\else\hfil\fi&&\thickspace\hfil
  $\m@th\scriptstyle##$\hfil\crcr
}{%
  \crcr\egroup\egroup\,%
}
\makeatother

\NewDocumentCommand{\ts}{O{c} e{^?_}}{
  \begin{smallermatrix}[#1]
  \mathstrut\IfValueT{#2}{#2} \\
  \mathstrut\IfValueT{#3}{#3} \\
  \mathstrut\IfValueT{#4}{#4}
  \end{smallermatrix}%
}

\newcommand{\revise}[1]{{\color{black}{#1}}}

\journal{Arxiv}

\begin{document}

\begin{frontmatter}



\title{Efficient and accurate nonlinear model reduction via first-order empirical interpolation}




\author[inst2]{Ngoc Cuong Nguyen}
\author[inst2]{Jaime Peraire}

\affiliation[inst2]{organization={Center for Computational Engineering, Department of Aeronautics and Astronautics, Massachusetts Institute of Technology},
            addressline={77 Massachusetts
Avenue}, 
            city={Cambridge},
            state={MA},
            postcode={02139}, 
            country={USA}}
            

\begin{abstract}
We present a model reduction approach that extends the original empirical interpolation method to enable  accurate and efficient reduced basis approximation of parametrized nonlinear partial differential equations (PDEs). In the presence of nonlinearity, the Galerkin reduced basis approximation remains computationally expensive due to the high complexity of evaluating the nonlinear terms, which depends on the dimension of the truth approximation. The  empirical interpolation method (EIM) was proposed as a nonlinear model reduction technique to render the complexity of evaluating the nonlinear terms independent of the dimension of the truth approximation. \revise{We introduce a first-order empirical interpolation method (FOEIM) that makes use of the partial derivative information to construct an inexpensive and stable interpolation of the nonlinear terms. We propose two different FOEIM algorithms to generate interpolation points and basis functions.  We apply the FOEIM  to nonlinear elliptic PDEs and compare it to the Galerkin reduced basis approximation and the EIM. Numerical results are presented to demonstrate the performance of the three reduced basis approaches.} 
\end{abstract}







\end{frontmatter}


\section{Introduction}
\label{sec:intro}

Many physical systems in engineering and science are described by partial differential equations (PDEs). The design, optimization, control, and characterization of physical systems often require repeated, accurate and fast prediction of quantities of interest (QoIs). The evaluation of QoIs  demands numerical approximation of the underlying PDE by finite element (FE), finite difference (FD) and finite volume (FV) methods. The numerical approximation of PDEs yields  high-dimensional  systems of equations --- known as full order models (FOMs) --- which can be computationally expensive for complex physical processes. Model reduction methods seek to reduce the computational complexity of FOMs by constructing reduced order models (ROMs) in significantly lower dimensional spaces. Projection-based model reduction techniques have been widely used to construct ROMs in numerous applications such as fluid mechanics \cite{rowley04:_compressible_pod,LeGresley2000,Knezevic2011,Lorenzi2016,Yano2019,Yano2020,Blonigan2021,Yu2022,Ballarin2016,Carlberg2013,Du2022}, solid mechanics \cite{ARCME,Huynh2007a,Farhat2014,Farhat2015,Tiso2013}, electromagnetic \cite{Chen2010a,Vidal-Codina2018a,Pomplun2010}, optimization \cite{Manzoni2012,Qian2017}, inverse problems \cite{Hoang2013,Nguyen_SantaFE08,Galbally2010,Lieberman2010},  multiscale problems \cite{Nguyen2008},  optimal control \cite{Bader2016,Karcher2018}, and data assimilation \cite{Karcher2018a,Maday2013,Maday2015b}. \revise{The success of projecton-based ROMs without hyper-reduction is limited to FOMs with affine parameter dependence \cite{ARCME,prudhomme02:_reliab_real_time_solut_param} and low-order polynomial nonlinearities \cite{Knezevic2011,Nguyen2009b}. In these cases, very efficient ROMs can be developed by resolving the affine and non-linear terms into the sum of products of the basis functions and coefficients.} 

In the presence of strong nonlinearities,  projection-based ROMs become computationally expensive without an efficient treatment of the nonlinear terms \cite{Grepl2007a,Nguyen2008d}. A number of different approaches have been developed to deal with nonlinear PDEs in the context of model reduction.  One approach is linearization \cite{weile01:_reduc_krylov} and polynomial approximation \cite{phillips03:_weakly_nonlinearMOR}. However, inefficient representation of the nonlinear terms and fast exponential growth (with the degree of the nonlinear approximation order) of the computational complexity render these methods quite expensive, in particular for strong nonlinearities.\footnote{\revise{This refers to the presence of high-order polynomial or non-polynomial nonlinearities in the parametrized PDEs. Typically, $q$-degree polynomial terms result in $O(N^{q+1})$ operation count to assemble the reduced model due to the expansion of the nonlinear terms into the sum of products of the basis functions and coefficients. Here $N$ is the dimension of the reduced model. Non-polynomial terms may not be admitted to such expansion and may lead to a higher computational cost that depends on the dimension of the full model.}} Another approach uses piecewise polynomials to approximate nonlinear terms \cite{rewienski03:_trajectory_approach}. However, there are still many nonlinear functions that may not be approximated well by using low degree piecewise polynomials unless there are very many constituent polynomials.

 
An efficient model reduction technique for PDEs with nonaffine parameter dependence was first proposed in \cite{Barrault2004a} that led to the development of the empirical interpolation method for constructing coefficient-function approximation of nonaffine terms. Shortly later, the empirical interpolation method was extended to develop efficient ROMs for nonlinear PDEs \cite{Grepl2007a}. Since the pioneer work \cite{Barrault2004a}, the EIM has been widely used to construct efficient ROMs of nonaffine and nonlinear PDEs for  different applications \cite{Grepl2007a,Nguyen2007,Drohmann2012,Manzoni2012,Hesthaven2014,Hesthaven2022,Chen2021}.  In \cite{Nguyen2008a}, the best-points interpolation method (BPIM) was developed to treat nonlinearity in FOMs and generate efficient ROMs for elliptic problems and convection-diffusion problems. In \cite{Galbally2010}, gappy POD, EIM, and BPIM were applied to a nonlinear combustion problem governed by an advection-diffusion-reaction PDE to enable the rapid solution of large-scale statistical inverse problems. In \cite{Kramer2019}, a lifting transformation method was proposed to reduce the complexity of evaluating nonlinear reduced order models. A new model reduction method for parametrized nonlinear PDEs is recently introduced in \cite{Yano2019a} to provide rapid evaluation of the nonlinear reduced order model via an empirical quadrature procedure. This method differs from the interpolation-then-integration approaches described earlier because it employs the sparse quadrature rules to directly approximate the nonlinear integrals.

The discrete empirical interpolation method (DEIM) \cite{Chaturantabut2010} is a discrete variant of the empirical interpolation method. DEIM consider a collection of vectors arising from the spatial discretization of parameter-dependent functions or PDEs and select a subset of vectors and associated interpolation indices. DEIM has also been widely used to construct efficient ROMs of  nonlinear PDEs \cite{Maierhofer2022,Tiso2013,Yu2022,Kramer2019}. DEIM is closely related to missing point estimation  \cite{Astrid2008} in the sense that both methods employ a small selected set of spatial grid points to avoid evaluation of the expensive inner products at every time step that are required to evaluate the nonlinearities. \revise{However, ROMs via (D)EIM have been shown to suffer from instabilities in certain situations \cite{Peherstorfer2020}. Adaptation \cite{Eftang2012b} and localization \cite{Peherstorfer2014} of the low-dimensional subspaces have been proposed to improve the stability of ROMs via empirical interpolation. Oversampling uses more interpolation points than basis functions so that the nonlinear terms are approximated via least-squares regression rather than via interpolation \cite{Peherstorfer2020}.  Oversampling methods such as  gappy POD \cite{everson95karhunenloeve,willcox06:_gappy}, missing point estimation \cite{Astrid2008,Zimmermann2016}, Gauss-Newton with approximated tensors \cite{Carlberg2013}, and generalized EIM (GEIM) \cite{Argaud2017} can provide more stable and accurate approximations than empirical interpolation especially when the samples are perturbed due to noise turbulence, and numerical inaccuracies.} 





The original empirical interpolation method (EIM) \cite{Barrault2004a,Grepl2007a}  was proposed as an efficient model reduction technique to render the complexity of evaluating the nonlinear terms independent of the dimension of the truth approximation. The main idea is to replace any nonlinear term with a reduced basis expansion expressed as a linear combination of pre-computed basis functions and parameter-dependent coefficients. The coefficients are determined efficiently by an inexpensive and stable interpolation procedure. In the empirical interpolation method, the basis functions are instances of the nonlinear function at $N$ parameter points in a sample set $S_N$. Therefore, the number of basis functions and interpolation points can not exceed $N$. In order to improve the approximation accuracy, we must increase $N$ at the expense of increasing the offline cost because we need to evaluate the FOM for all parameter points in $S_N$.

 
We seek to improve the accuracy without increasing the size of the sample set. To this end, we employ the partial derivatives of the nonlinear function evaluated at parameter points in $S_N$ to construct additional interpolating  functions and interpolation points. The resulting method is called the first-order empirical interpolation method (FOEIM) to distinguish itself from the EIM that does not use first-order partial derivatives. The  proposed method is applied to  nonlinear elliptic PDEs and compared to both the  Galerkin reduced basis approximation and the EIM. Numerical results are presented to assess the performance of the three reduced basis approaches.


Indeed, Hermitian spaces built upon sensitivity derivatives of the field variable with respect to the parameter ~\cite{Guillaume97} or, more generally, Lagrange-Hermitian spaces~\cite{ito98:_reduc_basis_method_contr_probl_gover_pdes} were long considered for the RB approximation of parametrized PDEs. \revise{We emphasize that our method does not require {\em sensitivity} derivatives of the field variable with respect to the parameters, which are often more expensive to compute than the field variable itself because they are obtained by differentiating  the underlying PDEs with respect to the parameters and solving the resulting PDEs. In fact, our method requires {\em partial} derivatives of the nonlinear terms with respect to the field variable and the parameters, which are  inexpensive to compute if the field variable is already computed. This is because evaluating the partial derivatives of the nonlinear terms has a similar cost as evaluating the nonlinear terms.}


The paper is organized as follows. In Section 2, we introduce the first-order empirical interpolation method.  The model problem, reduced basis approximations, and computational considerations for nonlinear elliptic problems and nonlinear diffusion equations are then discussed in Sections 3 and 4, respectively. Numerical results are included in each section in order to assess our method. Finally, in Section 5, we make a number of concluding remarks on the results as well as future work.

\section{First-order empirical interpolation}

\subsection{The interpolation problem}

Let $\Omega \subset  \mathbb{R}^d$ be the physical domain in which the spatial point $\bm x$ resides. Let $\mathcal{D} \subset \mathbb{R}^P$ be the parameter domain in which our $P$-tuple parameter point $\bm \mu$ resides. Let $u(\bm x, \bm \mu) \in L^\infty (\Omega \times \mathcal{D})$ be a parameter-dependent function of sufficient regularity. We consider the problem of approximating a nonlinear function $g(u(\bm x, \bm \mu), \bm \mu)$ by a reduced basis expansion $g_M(\bm x, \bm \mu)$, where $g(u, \bm \mu)$ is generally nonlinear with respect to the first argument. We assume that the partial derivatives, $g'_u(u, \bm \mu) \equiv \partial g(u, \bm \mu)/\partial u$ and $g'_{\bm \mu}(u, \bm \mu) \equiv \partial g(u, \bm \mu)/\partial \bm \mu$, are bounded everywhere in $\Omega \times \mathcal{D}$.

We assume that we are given a  set of interpolating functions $W_M^g = \mbox{span} \{\psi_m(\bm x), 1 \le m \le M\}$ and a set of interpolation points $T_M = \{\widehat{\bm x}_1, \ldots, \widehat{\bm x}_M\}$. The RB expansion $g_M(\bm x, \bm \mu)$ is expressed as 
\begin{equation}
\label{eq1w}
g_M(\bm x, \bm \mu) = \sum_{m=1}^M \beta_{M,m}(\bm \mu) \psi_m(\bm x)   ,
\end{equation}
where the coefficients $\beta_{M,m}(\bm \mu), 1 \le m \le M,$ are found as the solution of the following linear system 
\begin{equation}
\label{eq2w}
\sum_{m=1}^M  \psi_m(\widehat{\bm x}_k)   \beta_{M,m}(\bm \mu) = g(u(\widehat{\bm x}_k, \bm \mu), \bm \mu), \quad 1 \le k \le M .
\end{equation}
It is convenient to compute the coefficient vector $\bm \beta_M(\bm \mu)$ as follows
\begin{equation}
\label{eqcoeff}
\bm \beta_M(\bm \mu) = \bm B^{-1}_M \bm b_M(\bm \mu),
\end{equation}
where $\bm B_M \in \mathbb{R}^{M \times M}$ has entries $B_{M,km} =  \psi_m(\widehat{\bm x}_k)$ and $\bm b_M(\bm \mu) \in \mathbb{R}^M$ has entries $b_{M,k}(\bm \mu) =  g(u(\widehat{\bm x}_k, \bm \mu), \bm \mu)$. Since $g_M(\widehat{\bm x}_k, \bm \mu) = g(\widehat{\bm x}_k, \bm \mu), 1 \le k \le M,$ the RB expansion $g_M(\cdot, \bm \mu)$ is an interpolant of $g(\cdot,\bm\mu)$ over $M$ interpolation points. We define the associated error as
\begin{equation}
\label{eq4w}
\varepsilon_M(\bm \mu) = \|g(u(\bm x, \mu), \bm \mu) - g_M(\bm x, \bm \mu)\|_{L^\infty(\Omega)} ,    
\end{equation}
which is a measure of the approximation  accuracy for any given $\bm \mu \in \mathcal{D}$. The complexity of computing the coefficient vector  $\bm \beta_M(\bm \mu)$ in (\ref{eqcoeff}) for any given $\bm \mu$ is $O({M^2})$  because the matrix $\bm B_M^{-1}$ is pre-computed and stored.

The approximation accuracy depends critically on both the interpolating subspace $W_M^g$ and the interpolation point set $T_M$.   In what follows, we review the original empirical interpolation and introduce a first-order empirical interpolation for constructing $W_M^g$ and $T_M$.


\subsection{Empirical interpolation method}
\label{section2.2}

The empirical interpolation method was first proposed in \cite{Barrault2004a} to enable efficient RB approximation of PDEs with nonaffine parameter dependence and subsequently extended to nonlinear parametrized PDEs \cite{Grepl2007a}. We follow the empirical interpolation procedure described in \cite{Maday2008} to construct $W_M^g$ and $T_M$. We assume that we are given a sample set $S_N = \{\bm \mu_1 \in \mathcal{D}, \ldots, \bm \mu_N \in \mathcal{D} \}$. We then introduce two separate RB spaces 
\begin{subequations}
\label{EIMspaces}
\begin{alignat}{2}
W_N^u & = \mbox{span} \{\zeta_n(\bm x) \equiv u(\bm x, \bm \mu_n), 1 \le n \le N\}, \\
W_N^g  & = \mbox{span} \{\xi_n(\bm x) \equiv g(\zeta_n(\bm x), \bm \mu_n), 1 \le n \le N\} .
\end{alignat}
\end{subequations}
We assume that the dimension of these two RB spaces is equal to $N$. First, we find
\begin{equation}
j_1 = \arg \max_{1 \le j \le N} \|\xi_j  \|_{L^\infty(\Omega)},  
\end{equation}
and set
\begin{equation}
\widehat{\bm x}_1 = \arg \sup_{\bm x \in \Omega} |\xi_{j_1}(\bm x)|, \qquad \psi_1(\bm x) = \xi_{j_1}(\bm x)/\xi_{j_1}(\widehat{\bm x}_1).      
\end{equation}
Then for $M = 2, \ldots, N$, we solve the linear systems
\begin{equation}
\sum_{m=1}^{M-1}  \psi_m(\widehat{\bm x}_k)   \sigma_{nm} = \xi_n(\widehat{\bm x}_k), \quad 1 \le k \le M-1 ,
\end{equation}
for $n = 1,\ldots, N$; we find
\begin{equation}
j_{M} = \arg \max_{1 \le n \le N} \| \xi_n(\bm x) - \sum_{m=1}^{M-1}  \sigma_{nm} \psi_m(\bm x)\|_{L^\infty(\Omega)},  
\end{equation}
and set
\begin{equation}
\widehat{\bm x}_M = \arg \sup_{\bm x \in \Omega} |r_M(\bm x)|, \qquad \psi_M(\bm x) = r_M(\bm x)/r_M(\widehat{\bm x}_M)    ,   
\end{equation}
where the residual function $r_M(\bm x)$ is given by
\begin{equation}
r_M(\bm x) = \xi_{j_M}(\bm x) - \sum_{m=1}^{M-1}  \sigma_{j_M m} \psi_m(\bm x) .
\end{equation}
In essence, the interpolation point $\widehat{\bm x}_M$ and the basis function $\psi_M(\bm x)$ are the maximum point and the normalization of the residual function $r_M(\bm x)$ which results from the interpolation of $\xi_{j_M}(\bm x)$ by using the previous interpolation point set $T_{M-1}$ and basis set $W_{M-1}^g$. \revise{Note that the ordering of the snapshot functions in the space $W_N^g$ does not  affect the interpolation points and basis functions. In order words, the EIM yields exactly the same interpolation points and basis functions regardless of the ordering of the functions in $W_M^g$.}

The EIM procedure constructs nested sets of interpolation points $T_M, 1 \le M \le N,$ and nested subspaces $W_M^g = \mbox{span} \{\psi_m, 1 \le m \le M\}, 1 \le M \le N$. It is shown in \cite{Maday2008} this construction of the interpolation points  and the basis functions is well-defined, meaning that the basis functions are linearly independent and, in particular, the matrix $B_{M,km} = \psi_m(\widehat{\bm x}_k)$ is invertible. Furthermore,  the interpolation is exact for all $\xi$ in $W_M^g$. 


For any given sample set $S_N$, the number of interpolation points in $T_M$ and basis functions in $W_M^g$ can not exceed $N$. The choice of the sample set $S_N$  may have a crucial impact on the performance of the method. In order to achieve a desired accuracy, one may need to choose the sample set $S_N$ conservatively large enough. In the context of the reduced basis approach for parametrized PDEs, a large sample set $S_N$ will incur a high computational cost for the offline stage \revise{because we must compute $N$ solutions of the FOM to construct the function spaces defined in (\ref{EIMspaces}). Therefore, it is desirable to keep $S_N$ as small as possible, while being able to achieve the desired accuracy. For nonlinear elliptic PDEs considered herein as well as other nonlinear parametrized PDEs considered elsewhere \cite{Manzoni2012,Grepl2007a,Nguyen2007,Drohmann2012,Hesthaven2014,Hesthaven2022,Chen2021}, the computational complexity of ROMs via empirical interpolation during the online stage is $O(MN^2 + N^3)$ per Newton iteration. While the online cost scales cubically with $N$, it scales linearly with $M$. Therefore, we usually use $M > N$ to make the approximation of the nonlinear terms highly accurate, thereby obtaining  stable and accurate ROMs. This can only happen if we use a larger sample set $S_M = \{\bm \mu_1 \in \mathcal{D}, \ldots, \bm \mu_M \in \mathcal{D} \}$ so that $S_N \subset S_M$ to construct the function space $W_M^g = \mbox{span} \{\xi_m(\bm x) \equiv g(u(\bm x, \bm \mu_m), \bm \mu_m), 1 \le m \le M\} $, which  requires $M$ solutions of the FOM in the offline stage.}

\subsection{First-order empirical interpolation method}
\label{section2.3}


The main idea of the first-order empirical interpolation method is based on the first-order Taylor expansion
\begin{equation}
g_n^{(1)}(u, \bm \mu) = g(\zeta_n, \bm \mu_n) + \frac{\partial g(\zeta_n, \bm \mu_n)}{\partial u} (u - \zeta_n) + \frac{\partial g(\zeta_n, \bm \mu_n)}{\partial \bm \mu} \cdot (\bm \mu - \bm \mu_n)    
\end{equation}
for $1 \le n \le N$. Each $g_n^{(1)}(u, \bm \mu)$ is a first-order approximation to $g(u, \bm \mu)$ at  $(\zeta_n, \bm \mu_n)$. As $g(\zeta_n, \bm \mu_n)$, is a zero-order approximation to $g(u, \bm \mu)$,  the original empirical interpolation method only uses zero-order approximations to construct the interpolation points and interpolating functions. In order to improve the method, we  use the first-order partial derivatives to construct additional interpolation points and basis functions as follows. 

\revise{In addition to $W_N^u$ and $W_N^g$ defined in (\ref{EIMspaces}), we require $\frac{\partial g(\zeta_n, \bm \mu_n)}{\partial u}$ and $\frac{\partial g(\zeta_n, \bm \mu_n)}{\partial \bm \mu}$ for $1 \le n \le N$. These partial derivatives are inexpensive to compute if $\zeta_n, 1 \le n \le N,$ are already computed. In the context of this paper, the $\zeta_n$ are numerical solutions of the parametrized PDEs at the  parameter points in the sample set $S_N$, which can be computationally expensive. Once the $\zeta_n$ are computed, evaluating the partial derivatives of $g$ can have a similar cost as evaluating $g$ itself. We then introduce
\begin{equation}
\label{eqrho16}
\vartheta_{(n-1)N + k}(\bm x) = \frac{\partial g(\zeta_n(\bm x), \bm \mu_n)}{\partial u} (\zeta_k(\bm x) - \zeta_n(\bm x)), 
\end{equation}
and
\begin{equation}
\label{eqrho17}
\vartheta_{N^2 + (n-1)N + k}(\bm x) = \frac{\partial g(\zeta_n(\bm x), \bm \mu_n)}{\partial \bm \mu} \cdot (\bm \mu_k - \bm \mu_n),     
\end{equation}
for $1 \le k, n \le N$. Although there are $2N^2$ functions $\vartheta_n(\bm x), 1 \le n \le 2N^2,$ they are not linearly independent because there are $2N$ zero functions in (\ref{eqrho16}) and (\ref{eqrho17}) corresponding to $k = n$. As a result, there are at most $2(N^2-N)$ non-zero functions. Furthermore, if $g$ does not explicitly depend on $\bm \mu$ then the functions  in (\ref{eqrho17}) are zero because we have $\frac{\partial g(\zeta_n, \bm \mu_n)}{\partial \bm \mu} = 0$. In this case, there are at most $N^2 - N$ non-zero functions. Let $K$ be the number of linearly independent non-zero functions in the set $\{\vartheta_n(\bm x), 1 \le n \le 2N^2\}$. Without loss of generality, we denote those linearly independent functions as $\varrho_k(\bm x), 1 \le k \le K,$ and introduce the the following function space
\begin{equation}
\label{taylorspace}
W_K^{\partial g}   = \mbox{span} \{\varrho_k(\bm x), 1 \le k \le K \} .    
\end{equation}
The dimension of this function space will depend on the functional form of the nonlinear function $g$.} 

\revise{We pursue two different FOEIM algorithms to construct the basis functions and interpolation points. For the FOEIM Algorithm I,} the first $N$ interpolation points and basis functions are obtained by using the empirical interpolation procedure described in the previous subsection. The additional interpolation points and basis functions are calculated as follows. For $M = N+1, \ldots, N + K$, we solve the linear systems
\begin{equation}
\sum_{m=1}^{M-1}  \psi_m(\widehat{\bm x}_k)   \sigma_{nm} = \varrho_n(\widehat{\bm x}_k), \quad 1 \le k \le M-1 ,
\end{equation}
for $n = 1,\ldots, K$; we find
\begin{equation}
j_{M} = \arg \max_{1 \le n \le K} \| \varrho_n(\bm x) - \sum_{m=1}^{M-1}  \sigma_{nm} \psi_m(\bm x)\|_{L^\infty(\Omega)},  
\end{equation}
and set
\begin{equation}
\widehat{\bm x}_M = \arg \sup_{\bm x \in \Omega} |r_M(\bm x)|, \qquad \psi_M(\bm x) = r_M(\bm x)/r_M(\widehat{\bm x}_M)    ,   
\end{equation}
where the residual function $r_M(\bm x)$ is given by
\begin{equation}
r_M(\bm x) = \varrho_{j_M}(\bm x) - \sum_{m=1}^{M-1}  \sigma_{j_M m} \psi_m(\bm x) .
\end{equation}
\revise{The FOEIM Algorithm I constructs the first $N$ interpolation points and basis functions by applying the EIM  to the Lagrange space $W_N^{g}$, and  the next $K$ interpolation points and basis functions by applying the EIM to the Taylor space $W_K^{\partial g}$. Hence, the FOEIM Algorithm I includes partial derivative information only when $M > N$.}

\revise{For the FOEIM Algorithm II, we combine the Lagrange space $W_N^{g}$ and the Taylor space $W_K^{\partial g}$ into the following Lagrange-Taylor  space
\begin{equation}
\label{lagtaylorspace}
W_L^{{\rm LT}g}  = W_N^{g} \oplus W_K^{\partial g} \equiv \mbox{span}\{\varsigma_1, \ldots, \varsigma_L\} ,   
\end{equation}
where $L = N + K$ is the dimension of the Lagrange-Taylor space $W_L^{{\rm LT}g}$. We then apply the EIM  described in the previous subsection to $W_L^{{\rm LT}g}$ to obtain $L$  interpolation points and basis functions. First, we find
\begin{equation}
\label{eq23w}
j_1 = \arg \max_{1 \le j \le L} \|\varsigma_j  \|_{L^\infty(\Omega)},  
\end{equation}
and set
\begin{equation}
\widehat{\bm x}_1 = \arg \sup_{\bm x \in \Omega} |\varsigma_{j_1}(\bm x)|, \qquad \psi_1(\bm x) = \varsigma_{j_1}(\bm x)/\varsigma_{j_1}(\widehat{\bm x}_1).      
\end{equation}
For $M = 2, \ldots, L$, we solve the linear systems
\begin{equation}
\sum_{m=1}^{M-1}  \psi_m(\widehat{\bm x}_k)   \sigma_{lm} = \varsigma_l(\widehat{\bm x}_k), \quad 1 \le k \le M-1 , 1 \le l \le L,
\end{equation}
we then find
\begin{equation}
\label{argmax}
j_{M} = \arg \max_{1 \le l \le L} \| \varsigma_l(\bm x) - \sum_{m=1}^{M-1}  \sigma_{lm} \psi_m(\bm x)\|_{L^\infty(\Omega)}, 
\end{equation}
and set
\begin{equation}
\widehat{\bm x}_M = \arg \sup_{\bm x \in \Omega} |r_M(\bm x)|, \qquad \psi_M(\bm x) = r_M(\bm x)/r_M(\widehat{\bm x}_M)    ,   
\end{equation}
where the residual function $r_M(\bm x)$ is given by
\begin{equation}
\label{eq28w}
r_M(\bm x) = \varsigma_{j_M}(\bm x) - \sum_{m=1}^{M-1}  \sigma_{j_M m} \psi_m(\bm x) .
\end{equation}
Note that the interpolation points and basis functions generated by the FOEIM Algorithm II are independent of the ordering of the functions in the space $W_L^{{\rm LT}g}$. Therefore, unlike the FOEIM Algorithm I, the FOEIM Algorithm II  includes partial derivative information even when $M < N$.  
}

\revise{For any given parameter sample set $S_N$, the two FOEIM algorithms construct nested sets of interpolation points $T_M = \{\widehat{\bm x}_m\}_{m=1}^M, 1 \le M \le L,$ and nested subspaces $W_M^g = \mbox{span} \{\psi_m, 1 \le m \le M\}, 1 \le M \le L$. Although the same notations are used to label the interpolation points and basis functions constructed using the EIM and FOEIM algorithms, each algorithm generate different sets of interpolation points and basis functions. We point out the fact that the interpolation points and basis functions of the EIM are a subset of those of the FOEIM Algorithm I. However, this is not the case for the FOEIM Algorithm II. Indeed, the first $N$ interpolation points and basis functions of the FOEIM Algorithm II can be different from those of the EIM. For the same sample set $S_N$, the FOEIM algorithms improve the  EIM by leveraging the first-order partial derivatives to generate interpolation points and basis functions.}


\revise{The FOEIM algorithms have all the desirable properties of the  EIM. The algorithms are well-defined in the sense that the basis functions are linearly independent and the matrix $\bm B_M$ with entries $B_{M, ij} = \psi_j(\widehat{\bm x}_i), 1 \le i,j, \le M,$ is invertible. We follow \cite{Maday2008} to show an intermediate result:} 

\revise{
\begin{lemma}
Assume that $W_{M-1}^g = {\rm span} \: \{ \psi_1, \dots,\psi_{M-1} \}$ is of dimension $M-1$ and that $\bm B_{M-1}$ is invertible, then we have $v_{M-1} = v$ for any $v \in W_{M-1}^g$, where $v_{M-1}$ is the interpolant of $v$ as given below
\begin{equation}
v_{M-1} = \sum_{j=1}^{M-1} \beta_{M-1,j} \psi_{j} \ ,
\end{equation}
where the $\beta_{M-1,j}$ is the solution of
\begin{equation}
\sum_{j=1}^{M-1} \psi_j(\widehat{\bm x}_i) \beta_{M-1,j} = v(\widehat{\bm x}_i), \quad i = 1, \ldots, M-1 \ .
\end{equation}
In other words, the interpolation is exact for all v in $W^g_{M-1}$.
\end{lemma}
\begin{proof}
For $v \in W^g_{M-1}$, which can be expressed as $v(\bm x) = \sum_{m=1}^{M-1} \gamma_{M-1,m} \psi_m(\bm x)$, we consider $\bm x = \widehat{\bm x}_m, 1 \leq m \leq M-1,$ to arrive at $v(\widehat{\bm x}_m) = \sum_{j=1}^{M-1} \gamma_{M-1,j} \psi_j(\widehat{\bm x}_m), 1 \leq m \leq M-1$. It thus follows from the invertibility of $\bm B_{M-1}$ that $\bm \beta_{M-1} = \bm \gamma_{M-1}$; and hence $v_{M-1} = v$. 
\end{proof}

We  thus obtain the following theorem whose proof is  given in \cite{Maday2008}:
\begin{thm}
Assume that the dimension of the Lagrange-Taylor space $W_L^{{\rm LT}g}$ is $L$; then, for any $M \leq L$, the space $W_M^g = \mbox{span}\{\psi_1, \ldots, \psi_M\}$ is of dimension $M$. In addition, the matrix $\bm B_M$ is lower triangular with unity diagonal.   
\end{thm}
\begin{proof}
We shall proceed by induction. Clearly, $W_1 = {\rm span} \: \{ \psi_1 \}$ is of dimension 1 and the matrix $\bm B_1 = 1$ is invertible.  Next we assume that $W_{M-1}^g = {\rm span} \: \{ \psi_1, \dots, \psi_{M-1} \}$ is of dimension $M-1$ and the matrix $B^{M-1}$ is invertible; we must then prove ({\it i}) $W_M^g = \mbox{span}\{\psi_1, \ldots, \psi_M\}$ is of dimension $M$ and ({\it ii}) the matrix $\bm B_M$ is invertible. To prove ({\it i}), we note from our ``arg max'' construction (\ref{argmax}) and the assumption stated in  Theorem 1 that $\|r_M(\bm x)\|_{L^\infty(\Omega)} > 0$. Hence, if $\dim(W_M^g) \neq M$, we have $\varsigma_{j_M} \in W^g_{M-1}$ and thus $\|r_M(\bm x)\|_{L^\infty(\Omega)} = 0$ by Lemma 1; however, the latter contradicts $\|r_M(\bm x)\|_{L^\infty(\Omega)} > 0$.  To prove ({\em ii\/}), we just note from the
  construction procedure (\ref{eq23w})-(\ref{eq28w}) that $B_{{M},{i \: j}} = r_j(\widehat{\bm x}_i) /r_j (\widehat{\bm x}_j) = 0$
  for $i < j$; that $B_{{M},{i \: j}} = r_j(\widehat{\bm x}_i) /r_j (\widehat{\bm x}_j) = 1$ for $i =
  j$; and that $\left|B_{{M},{i \: j}} \right| = \left|r_j(\widehat{\bm x}_i) /r_j (\widehat{\bm x}_j) \right| \leq 1$ for $i > j$ since $\widehat{\bm x}_j = \arg  \:
  \max_{x \in \Omega} |r_j (\bm x)|, 1 \leq j \leq M$. Hence, $\bm B_{M}$ is
  lower triangular with unity diagonal.
\end{proof}
}

\revise{
This theorem implies that the procedure yields unique interpolation points and linearly independent basis functions as long as $M$ is less than or equal to the dimension of the function space used to construct the basis functions and the interpolation points. Furthermore, the procedure reorders members of the function space in such a way that $W_M^g = \mbox{span} \{\varsigma_{j_1}, \ldots, \varsigma_{j_M}\}$ = $\mbox{span} \{\psi_{1}, \ldots, \psi_{M}\}$. Hence, the procedure allows for selecting a subset of  basis functions from a larger set. The error analysis of the interpolation procedure involves the Lebesgue constant as follows
\begin{lemma}
Let $g_M(\bm x, \bm \mu)$ defined by (\ref{eq1w}) be the interpolant of the parametrized function $g(u(\bm x, \bm \mu), \bm \mu)$. The interpolation error (\ref{eq4w}) is bounded by
\begin{equation}
\label{EIMbound}
\varepsilon_M(\bm \mu) \le (1 + \Lambda_M)  \inf_{v_M \in W_M^g}  \|g(u(\bm x, \bm \mu), \bm \mu) - v_M \|_{L^\infty(\Omega)}, 
\end{equation}
where $\Lambda_M$ is the Lebesgue constant 
\begin{equation}
\Lambda_M = \sup_{\bm x \in \Omega} \sum_{j=1}^M \left|\sum_{m=1}^M \psi_m(\bm x) [\bm B_M]^{-1}_{mk} \right| .    
\end{equation}
\end{lemma}
\noindent
The last term in the right hand side of the above inequality is known as the best approximation error. This Lemma has been proven in \cite{Barrault2004a}. Furthermore, an upper-bound for the Lebesgue constant is $2^M-1$ \cite{Maday2008}.  
}


\revise{For nonlinear elliptic PDEs considered herein, the computational complexity of ROMs via first-order empirical interpolation during the online stage is also $O(MN^2 + N^3)$ per Newton iteration. Therefore, we  can improve the accuracy of ROMs by using a larger number of interpolation points and basis functions to approximate the nonlinear terms. Unlike the EIM procedure described earlier,  the FOEIM procedure requires $N$ solutions of the FOM in the offline stage even when we choose $M > N$ owing to the derivative-based function space $W_K^{\partial g}$.}


\revise{
\subsection{Empirical regression procedure}


Empirical regression extends empirical interpolation by using more interpolation points than basis functions. Specifically, we  use $N$ basis functions from the subspace $W_N^g = \mbox{span} \{\psi_1, \ldots, \psi_N\}$ and $M$ interpolation points from $T_M = \{\widehat{\bm x}_1,\ldots,\widehat{\bm x}_M\}$ to define an approximation $g_{NM}(\bm x, \bm \mu)$ to $g(u(\bm x, \bm \mu), \bm \mu)$ as follows 
\begin{equation}
\label{eq1wr}
g_{NM}(\bm x, \bm \mu) = \sum_{n=1}^N \beta_{N,n}(\bm \mu) \psi_n(\bm x)   .
\end{equation}
Here the coefficients $\beta_{N,n}(\bm \mu), 1 \le n \le N,$ are found as the solution of the following least-squares problem
\begin{equation}
\label{eq2wr}
\bm \beta_{N}(\bm \mu) = \arg \min_{\bm \gamma_N \in \mathbb{R}^N} \sum_{m=1}^M \left( \sum_{n=1}^N  \psi_n(\widehat{\bm x}_m)   \gamma_{N,n} - g(u(\widehat{\bm x}_m, \bm \mu), \bm \mu) \right)^2 .
\end{equation}
This least-squares problem reduces to the linear system (\ref{eq2w}) for $M=N$. In this case, $g_{NM}(\bm x, \bm \mu)$ interpolates $g(u(\bm x, \bm \mu), \bm \mu)$ exactly  at the interpolation points. For $M > N$, $g_{NM}(\bm x, \bm \mu)$ is the best fit in the 2-norm to $g(u(\bm x, \bm \mu), \bm \mu)$ at the interpolation points. In this case, it is convenient to compute the coefficient vector $\bm \beta_N(\bm \mu)$ as follows
\begin{equation}
\label{eqcoeffls}
\bm \beta_N(\bm \mu) = \bm C_{NM} b_M(\bm \mu),
\end{equation}
where $\bm C_{NM} = \left( \bm B_{NM} \bm B^T_{NM} \right)^{-1} \bm B_{NM}$, $\bm B_{NM} \in \mathbb{R}^{N \times M}$ has entries $B_{NM,nm} =  \psi_n(\widehat{\bm x}_m)$ and $\bm b_M(\bm \mu) \in \mathbb{R}^M$ has entries $b_{M,m}(\bm \mu) =  g(u(\widehat{\bm x}_m, \bm \mu), \bm \mu)$. The cost of evaluating $\bm \beta_N(\bm \mu)$ in (\ref{eqcoeffls}) is $O(MN)$ since we compute and store $\bm C_{NM}$ in the offline stage. We use the two FOEIM algorithms described earlier to generate interpolation points and basis functions for the empirical regression.
}

\revise{Empirical regression belongs to the class of oversampling methods like  gappy POD \cite{everson95karhunenloeve,willcox06:_gappy}, missing point estimation \cite{Astrid2008,Zimmermann2016}, Gauss-Newton with approximated tensors \cite{Carlberg2013}, and generalized EIM (GEIM) \cite{Argaud2017}. Empirical regression can provide more stable and accurate approximations than empirical interpolation for problems with noisy data \cite{Peherstorfer2020}. For model reduction of  nonlinear elliptic PDEs discussed in the next section, the online complexity of ROMs via empirical regression is also $O(MN^2 + N^3)$ per Newton iteration.} 


\subsection{A Gaussian parametrized function}

\revise{
We consider a parameter-dependent function $u(\bm x, \bm \mu) = \frac{1}{\sqrt{  (x_{1} - \mu_{1})^2 + (x_{2} -
    \mu_{2})^2 }}$ and a Gaussian parameterized function $g(u, \bm \mu) = \exp (-0.01 u^2)$ for $\bm x  \in \Omega \equiv \: (0,1) \: ^2$ and $\bm \mu \in \mathcal{D} \equiv [-1, -0.01]^2$. This
example is modified from the one in \cite{Grepl2007a}. Note that the partial derivatives $\partial g(u, \bm \mu)/\partial \bm \mu$ are zero for this particular function.  We choose for $S_{N_{\max}}$ a
deterministic grid of ${N_{\max}}= 8 \times 8$ parameter points over $\mathcal{D}$, and generate a sequence of nested sample sets $S_N \in S_{N_{\max}}$ for $N = 2 \times 2, 3 \times 3, \ldots, 8 \times 8$ by using the following logarithm distribution 
\begin{equation}
y(x) = a + (b-a)(1-\exp(-\alpha(x-a)/(b-a)))/(1-\exp(-\alpha)), 
\end{equation}
for $x \in [a, b]$, where $a = -1, b = -0.01, \alpha = 3$. The function $y(x)$ maps a uniform grid into a logarithmic grid such that the resulting grid is clustered toward $b$. As shown in Figure~\ref{fig0}, the parameter points in $S_{N_{\max}}$ are mainly distributed around the corner $(-0.01,-0.01)$ of the parameter domain. 
}

\begin{figure}[htbp]
	\centering
 \includegraphics[width=0.7\textwidth]{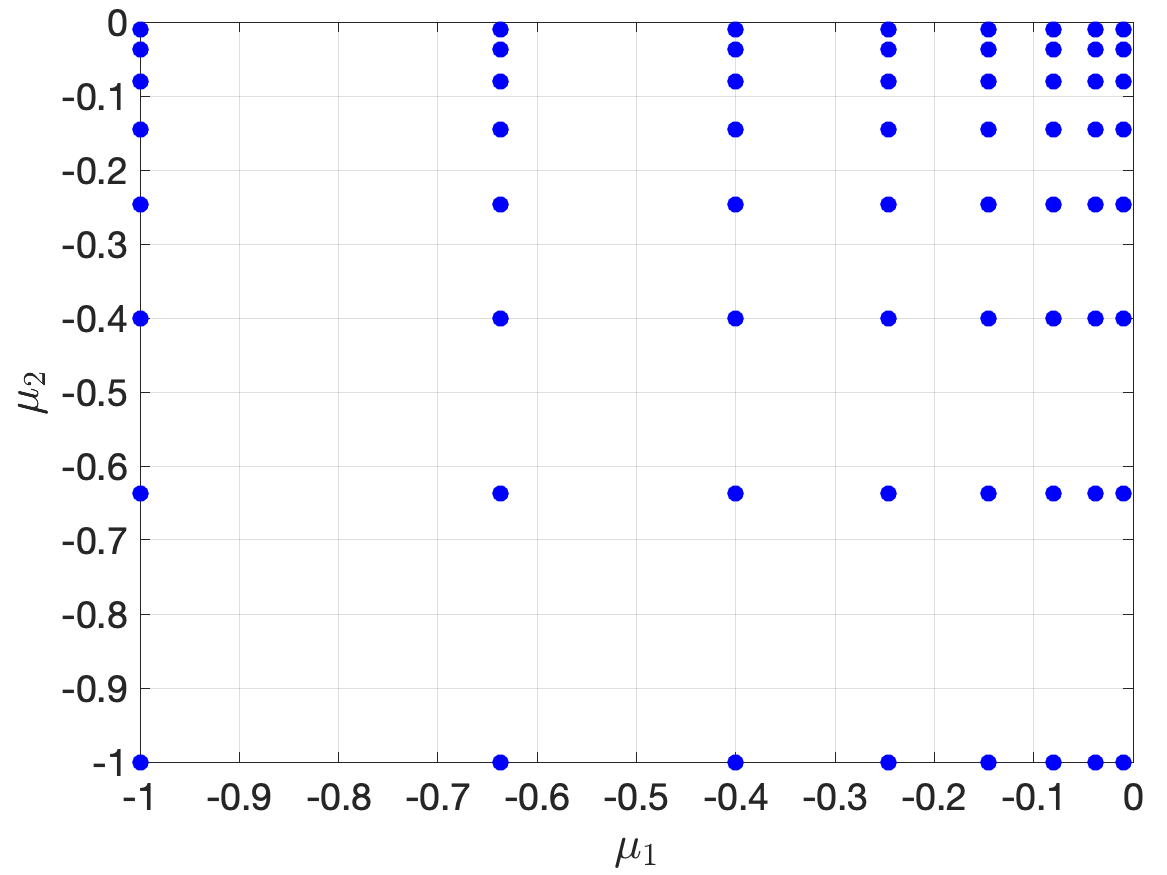}
 \caption{Parameter sample set $S_{N_{\max}}$ for the Gaussian parametrized function.}
	\label{fig0}
\end{figure}

\revise{For the results reported herein, FOEIM-I and FOEIM-II refer to FOEIM Algorithm I and FOEIM Algorithm II, respectively. Furthermore, FOERM-I (respectively, FOERM-II) refers to empirical regression which uses the interpolation points and the basis functions computed by FOEIM-I (respectively, FOEIM-II). We consider three different values of $M$, namely, $M = N, M = 2N$, and $M = 3N$, for those methods.  The interpolation points are plotted in Figure~\ref{fig1} for $M = 2N = 128$. We note that the interpolation points are largely allocated around
the origin $(0,0)$ of the physical domain $\Omega$. This is because $u(\bm x, \bm \mu)$ varies most significantly at  $\bm x=(0,0)$. The two FOEIM algorithms yield similar distributions of the interpolation points in the physical domain.}

\begin{figure}[htbp]
	\centering
	\begin{subfigure}[b]{0.49\textwidth}
		\centering
		\includegraphics[width=\textwidth]{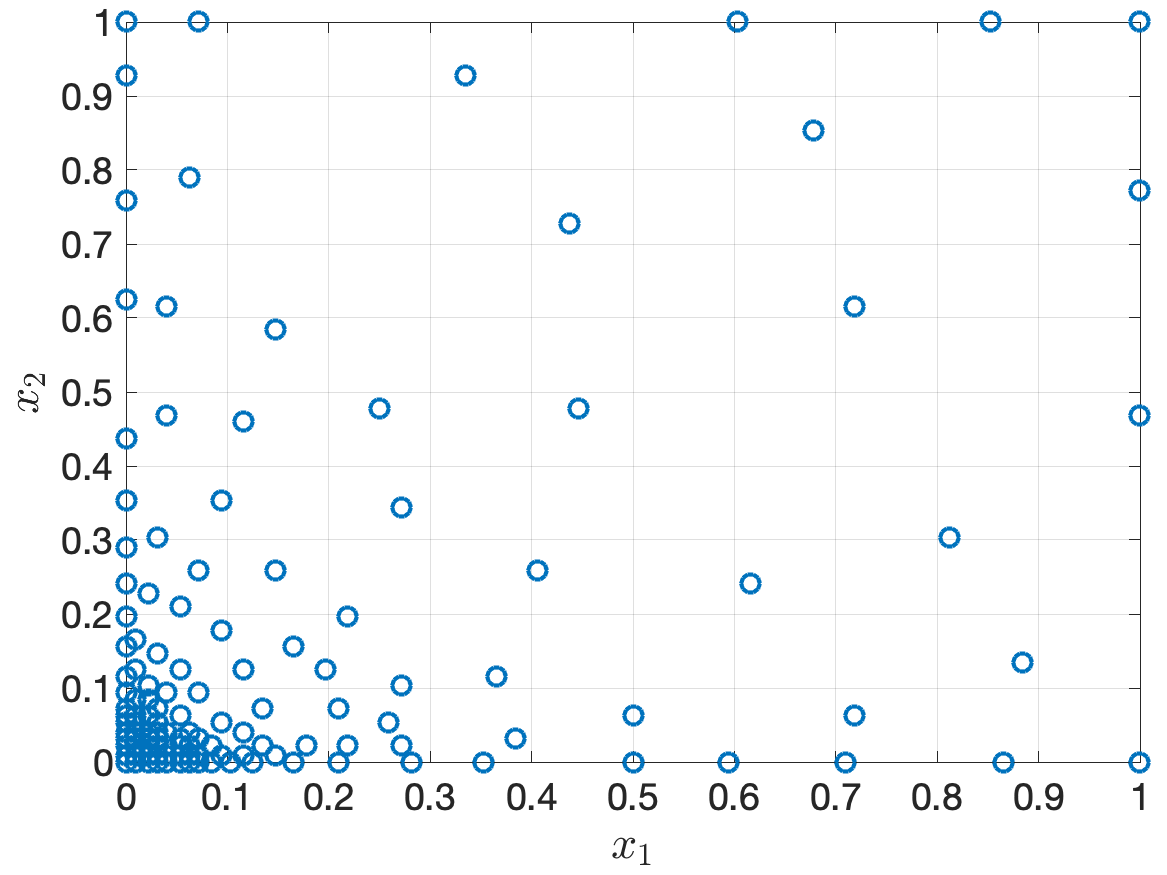}
		\caption{FOEIM-I.}
	\end{subfigure}
	\hfill
	\begin{subfigure}[b]{0.49\textwidth}
		\centering
		\includegraphics[width=\textwidth]{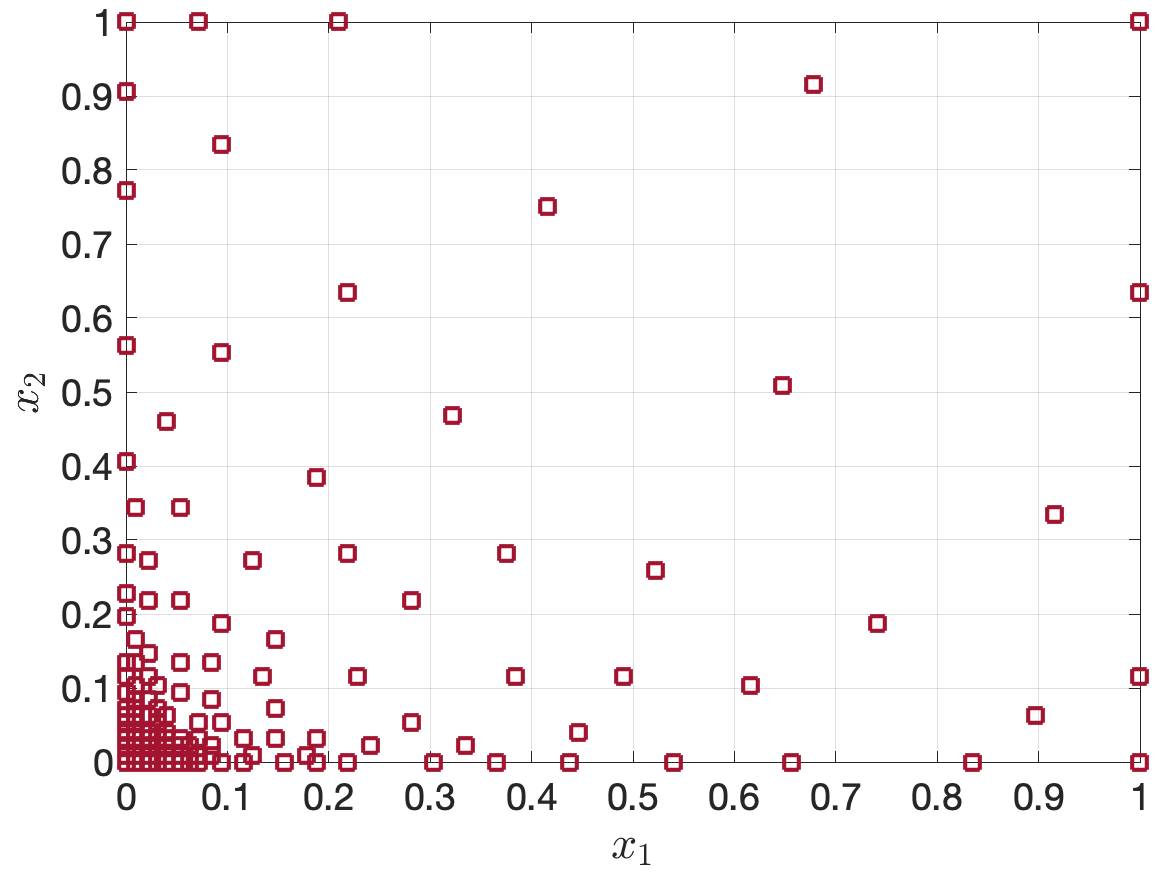}
		\caption{FOEIM-II.}
	\end{subfigure}
	\caption{Distribution of interpolation points in the phsyical domain for FOEIM-I and FOEIM-II for $M = 2N = 128$.}
	\label{fig1}
\end{figure}
 
\revise{We now introduce a uniform grid of size $N_{\rm
  Test} = 30 \times 30$ as a parameter test sample $S^g_{\rm Test}$, and define 
\begin{equation}
\varepsilon_{M}^{\max} = \max_{\bm \mu \in S^g_{\rm Test}}    \|g(u(\bm x, \mu), \bm \mu) - g_M(\bm x, \bm \mu)\|_{L^\infty(\Omega)}
\end{equation}     
as the maximum interpolation error.  We display in Figure~\ref{fig2}  $\varepsilon_{M}^{\max}$ and $\Lambda_M$ as a function of $M$.  We observe that $\varepsilon_{M}^{\max}$ converge rapidly with $M$ while the Lebesgue constant grows slowly with $M$. These results are expected: although $g(u,\bm \mu)$ varies rapidly as $\bm \mu$ approaches $0$ and $\bm x$ approaches 0, $g(u,\bm \mu)$ is nevertheless quite smooth in the prescribed parameter domain $\mathcal{D}$. We see from Figure~\ref{fig2}(a) that  FOEIM-II yields considerably smaller errors than EIM for $M = N = 64$, which can be attributed to the use of partial derivatives. We also observe from  Figure~\ref{fig2}(c)-(e) that FOEIM-I and FOEIM-II yield similar errors for $M=2N$ and $M=3N$, while FOERM-II tends to yield smaller errors than FOERM-I. These results are also shown in Table \ref{gausstab1} and Table \ref{gausstab2}. We see from Table \ref{gausstab1} that increasing $M = N$ to $M = 2N$ and $M = 3N$ reduces the maximum interpolation error by several orders of magnitudes for FOEIM-I and FOEIM-II. We observe from Table \ref{gausstab2} that the maximum regression error drops sightly as we increase $M = N$ to $M=2N$ for FOERM-I and FOERM-II. However, increasing $M=2N$ to $M=3N$ does not really improve the regression error. These results imply that empirical interpolation performs better than empirical regression for the same number of interpolation points.}

\begin{figure}[hthbp]
	\centering
	\begin{subfigure}[b]{0.49\textwidth}
		\centering
		\includegraphics[width=\textwidth]{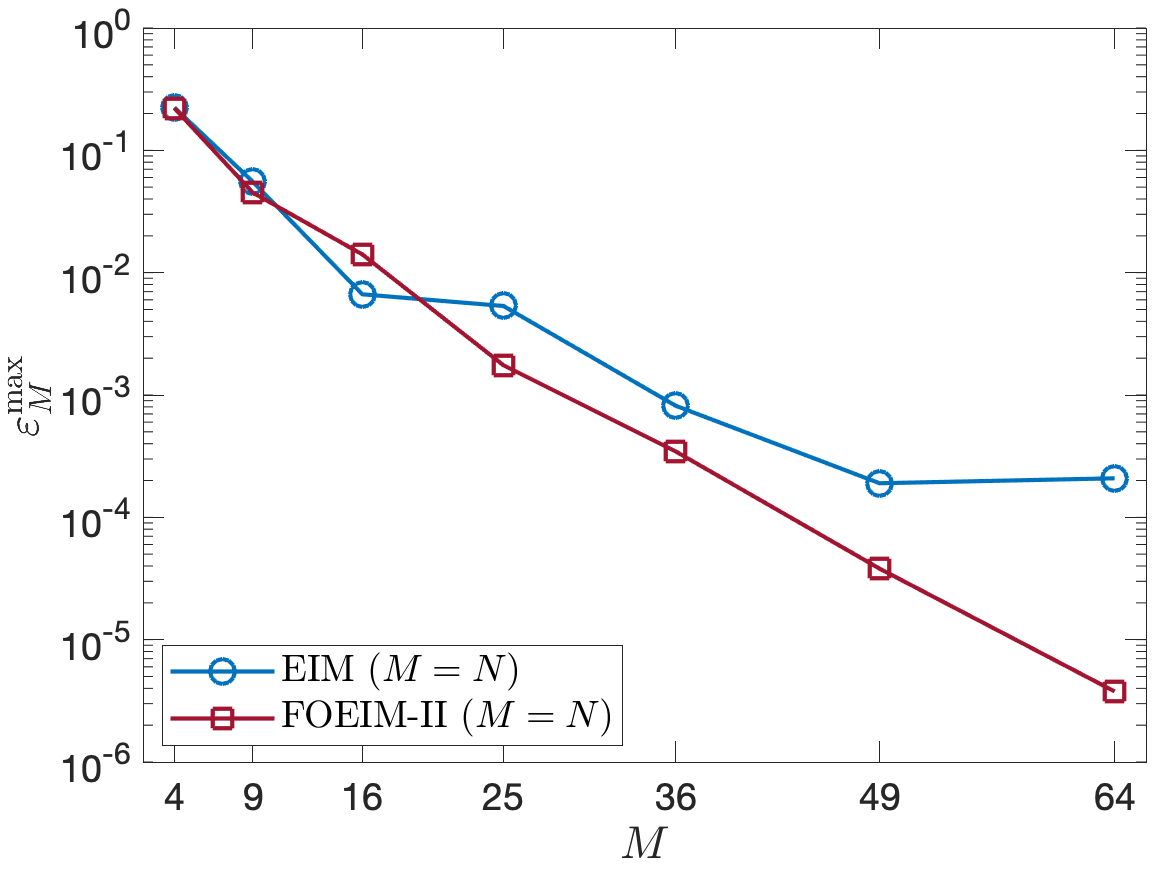}
		\caption{Maximum interpolation error.}
	\end{subfigure}
	\hfill
	\begin{subfigure}[b]{0.49\textwidth}
		\centering
		\includegraphics[width=\textwidth]{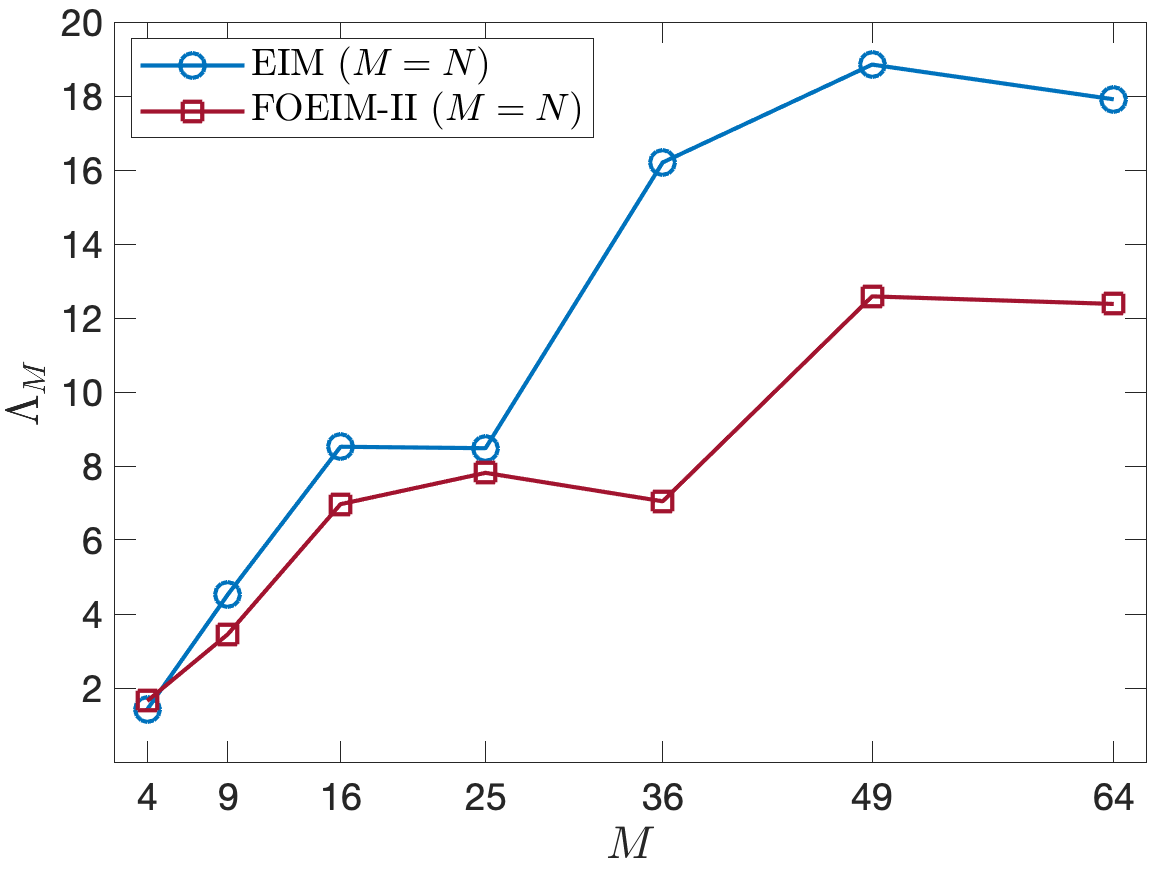}
		\caption{Lebesgue constant.}
	\end{subfigure}
        \begin{subfigure}[b]{0.49\textwidth}
		\centering
		\includegraphics[width=\textwidth]{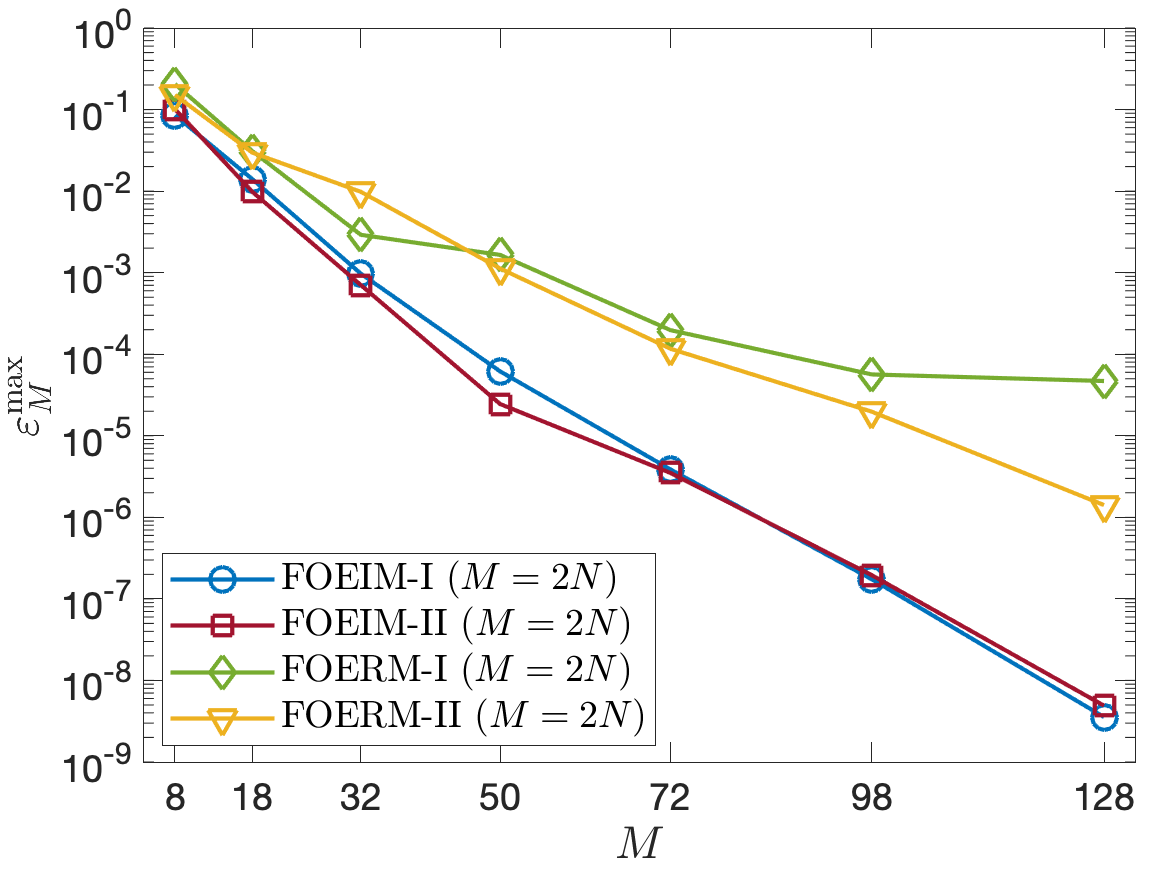}
		\caption{Maximum interpolation error.}
	\end{subfigure}
	\hfill
	\begin{subfigure}[b]{0.49\textwidth}
		\centering
		\includegraphics[width=\textwidth]{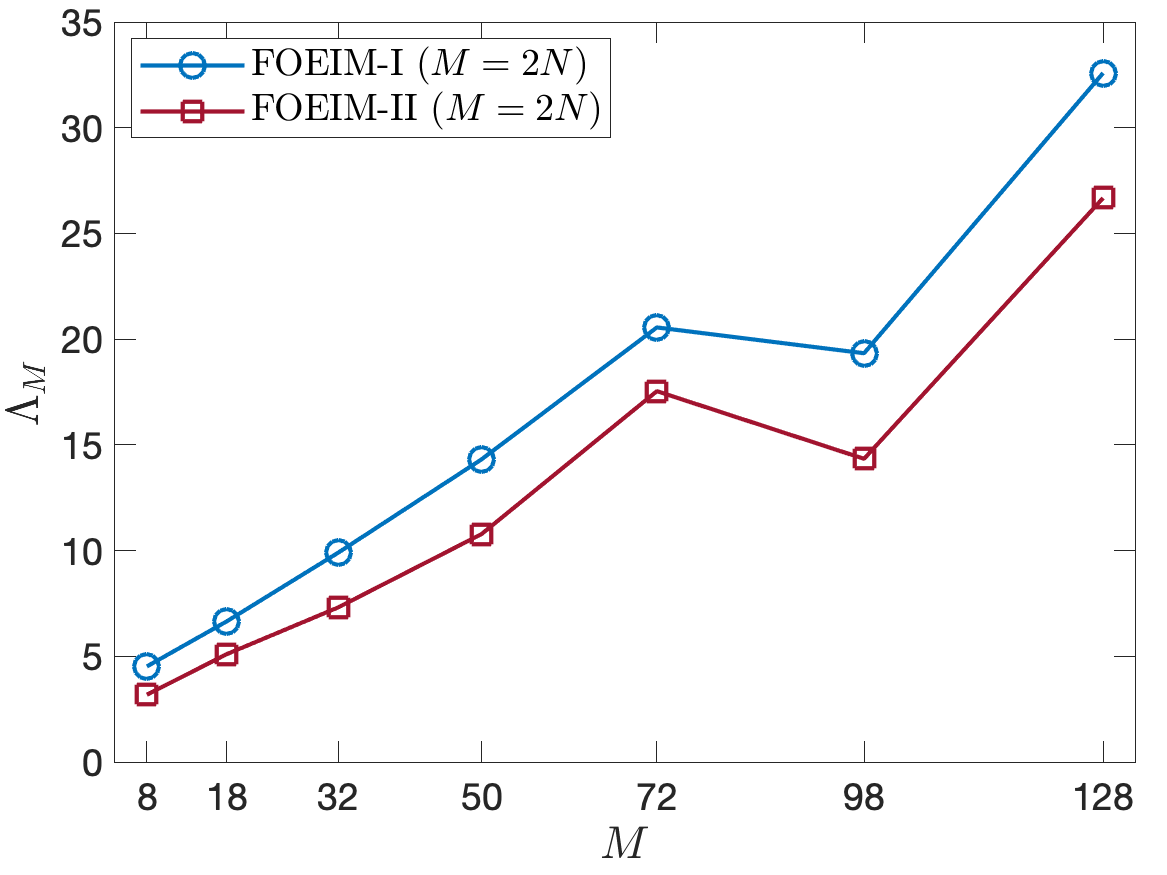}
		\caption{Lebesgue constant.}
	\end{subfigure}
        \begin{subfigure}[b]{0.49\textwidth}
		\centering
		\includegraphics[width=\textwidth]{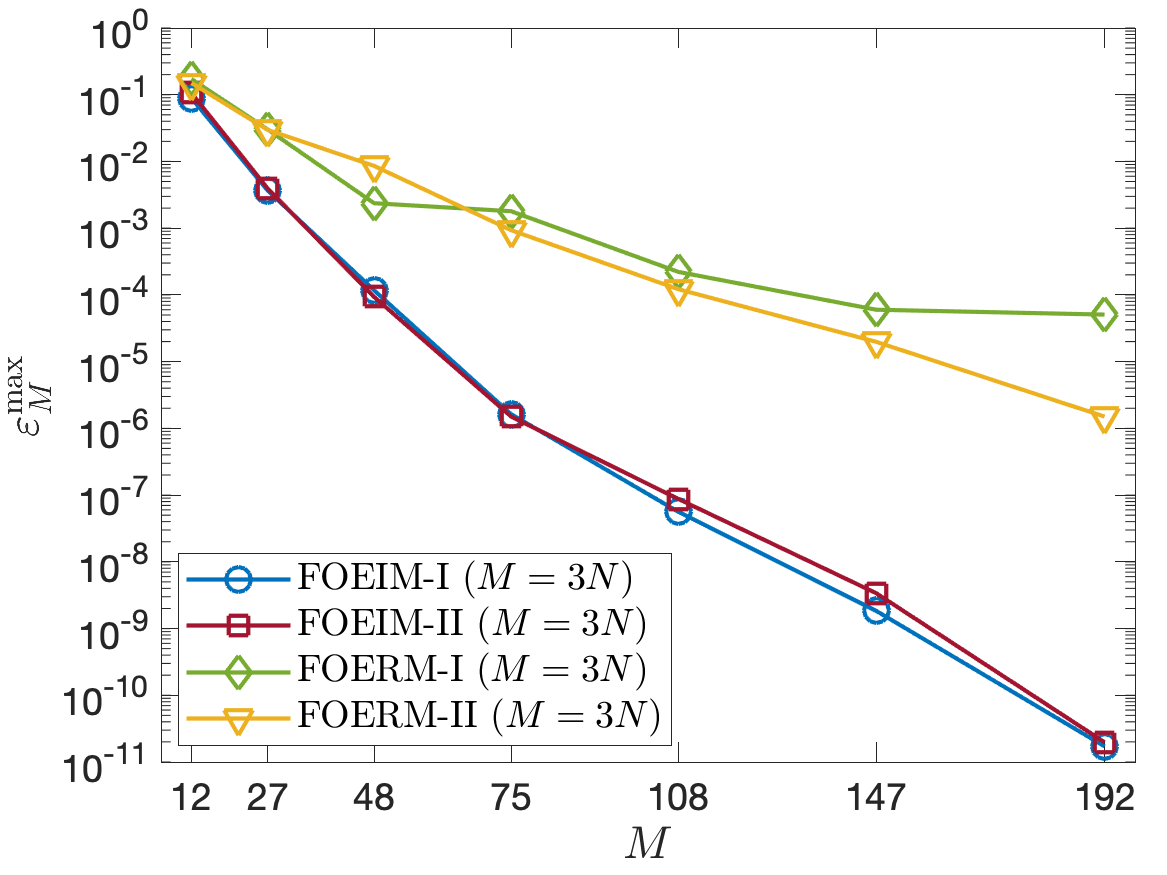}
		\caption{Maximum interpolation error.}
	\end{subfigure}
	\hfill
	\begin{subfigure}[b]{0.49\textwidth}
		\centering
		\includegraphics[width=\textwidth]{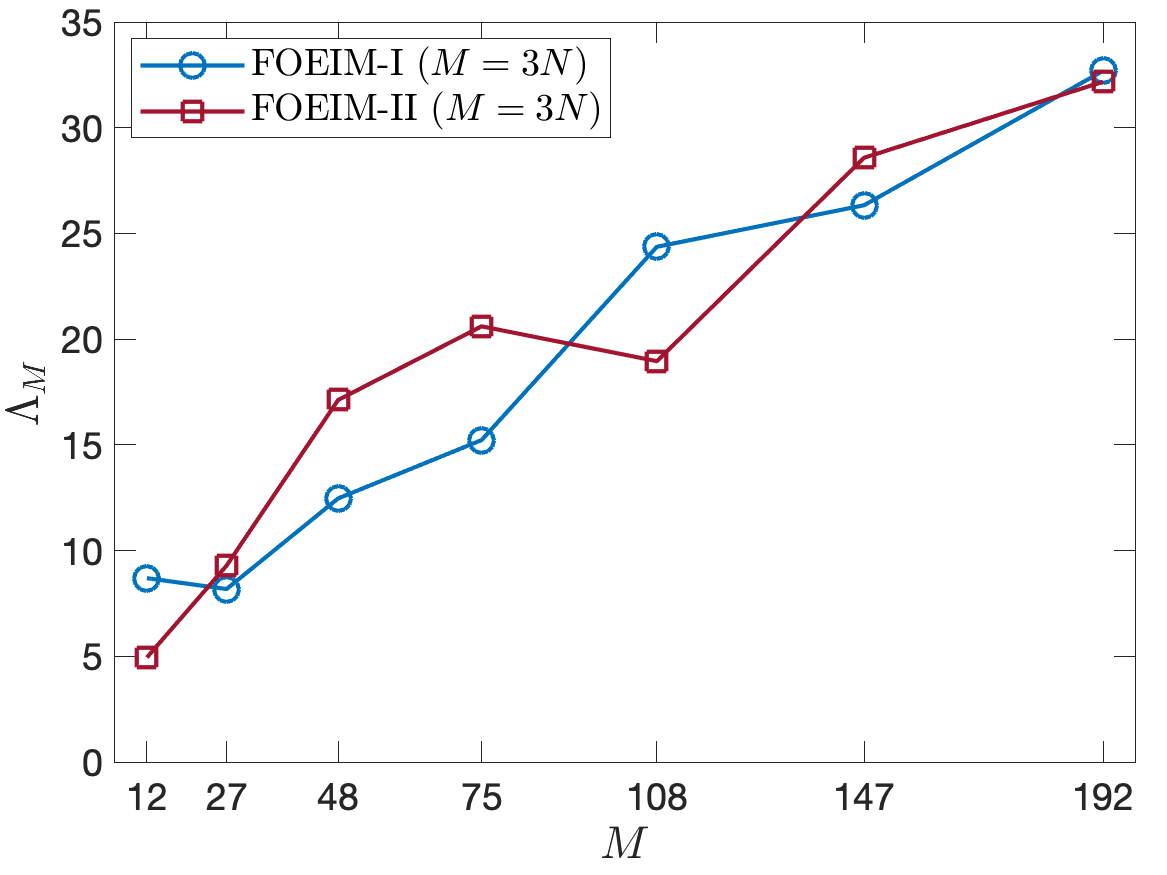}
		\caption{Lebesgue constant.}
	\end{subfigure}
	\caption{Convergence of the maximum interpolation error and the Lesbegue constant as a function of $M$ for  EIM, FOEIM-I, FOEIM-II, FOERM-I, FOERM-II.}
	\label{fig2}
\end{figure}
 
\begin{table}[h]
\centering
\small
	\begin{tabular}{|c||c|c|c||c|c|c|}
		\cline{1-7}
  & FOEIM-I  & FOEIM-I & FOEIM-I & FOEIM-II  & FOEIM-II & FOEIM-II \\  
   $N$ & $M=N$ &  $M=2N$ & $M=3N$ & $M=N$ & $M=2N$ & $M=3N$ \\
		\cline{1-7}
4  &  2.24\mbox{e-}1  &  8.50\mbox{e-}2  &  8.90\mbox{e-}2  &  2.21\mbox{e-}1  &  1.01\mbox{e-}1  &  1.06\mbox{e-}1  \\  
 9  &  5.55\mbox{e-}2  &  1.40\mbox{e-}2  &  3.69\mbox{e-}3  &  4.51\mbox{e-}2  &  9.79\mbox{e-}3  &  3.95\mbox{e-}3  \\  
 16  &  6.65\mbox{e-}3  &  9.73\mbox{e-}4  &  1.17\mbox{e-}4  &  1.41\mbox{e-}2  &  7.05\mbox{e-}4  &  9.34\mbox{e-}5  \\  
 25  &  5.34\mbox{e-}3  &  6.16\mbox{e-}5  &  1.63\mbox{e-}6  &  1.75\mbox{e-}3  &  2.45\mbox{e-}5  &  1.51\mbox{e-}6  \\  
 36  &  8.16\mbox{e-}4  &  3.86\mbox{e-}6  &  5.60\mbox{e-}8  &  3.45\mbox{e-}4  &  3.50\mbox{e-}6  &  8.73\mbox{e-}8  \\  
 49  &  1.90\mbox{e-}4  &  1.74\mbox{e-}7  &  1.84\mbox{e-}9  &  3.81\mbox{e-}5  &  1.95\mbox{e-}7  &  3.38\mbox{e-}9  \\  
 64  &  2.08\mbox{e-}4  &  3.56\mbox{e-}9  &  1.72\mbox{e-}11  &  3.77\mbox{e-}6  &  4.91\mbox{e-}9  &  1.97\mbox{e-}11  \\   
		\hline
	\end{tabular}
	\caption{Maximum interpolation error $\varepsilon_M^{\rm max}$ as a function of $N$ and $M$ for FOEIM-I and FOEIM-II. Note that FOEIM-I is the same as EIM for $M=N$.} 
	\label{gausstab1}
\end{table}

\begin{table}[htbp]
\centering
\small
	\begin{tabular}{|c||c|c|c||c|c|c|}
		\cline{1-7}
  & FOERM-I  & FOERM-I & FOERM-I & FOERM-II  & FOERM-II & FOERM-II \\  
   $N$ & $M=N$ &  $M=2N$ & $M=3N$ & $M=N$ & $M=2N$ & $M=3N$ \\
		\cline{1-7}
4  &  2.24\mbox{e-}1  &  2.05\mbox{e-}1  &  1.72\mbox{e-}1  &  2.21\mbox{e-}1  &  1.51\mbox{e-}1  &  1.50\mbox{e-}1  \\  
 9  &  5.55\mbox{e-}2  &  3.09\mbox{e-}2  &  3.02\mbox{e-}2  &  4.51\mbox{e-}2  &  2.96\mbox{e-}2  &  2.98\mbox{e-}2  \\  
 16  &  6.65\mbox{e-}3  &  2.91\mbox{e-}3  &  2.37\mbox{e-}3  &  1.41\mbox{e-}2  &  9.81\mbox{e-}3  &  8.58\mbox{e-}3  \\  
 25  &  5.34\mbox{e-}3  &  1.65\mbox{e-}3  &  1.80\mbox{e-}3  &  1.75\mbox{e-}3  &  1.12\mbox{e-}3  &  9.34\mbox{e-}4  \\  
 36  &  8.16\mbox{e-}4  &  1.97\mbox{e-}4  &  2.21\mbox{e-}4  &  3.45\mbox{e-}4  &  1.17\mbox{e-}4  &  1.22\mbox{e-}4  \\  
 49  &  1.90\mbox{e-}4  &  5.65\mbox{e-}5  &  6.00\mbox{e-}5  &  3.81\mbox{e-}5  &  1.97\mbox{e-}5  &  1.98\mbox{e-}5  \\  
 64  &  2.08\mbox{e-}4  &  4.69\mbox{e-}5  &  5.07\mbox{e-}5  &  3.77\mbox{e-}6  &  1.41\mbox{e-}6  &  1.50\mbox{e-}6  \\  
		\hline
	\end{tabular}
	\caption{Maximum regression error $\varepsilon_M^{\rm max}$ as a function of $N$ and $M$ for FOERM-I and FOERM-II. Note that FOERM-I is the same as EIM for $M=N$.} 
	\label{gausstab2}
\end{table}
   
\section{Nonlinear elliptic equations}

\subsection{A model problem}
\label{section3.1}

We consider the following parametrized nonlinear elliptic PDE 
\begin{equation}
-\nabla^2u + \mu_1 \exp(\sin(\mu_2 u)) = 100 \sin(2\pi x_{1})\cos(2 \pi x_{2}), \quad \mbox{in } \Omega, 
\end{equation} 
 with homogeneous Dirichlet condition on the boundary $\partial \Omega$, where $\Omega = (0,1)^2$ and $\bm \mu  \in {\cal D} \equiv [1, 10]^2$.  The output of interest is the average of the field variable over the
physical domain. The weak formulation is then stated as: given $\bm \mu \in
{\cal D}$, find $s(\bm \mu) = \int_\Omega u(\bm \mu)$, where $u(\bm \mu) \in X \subset H_0^1(\Omega) \equiv \{v \in
H^1(\Omega) \mbox{ } | \mbox{ } v|_{\partial \Omega} = 0\}$ is the solution
of
\begin{equation}
a(u(\bm \mu), v) + \mu_1 \int_\Omega g(u(\bm \mu), \bm \mu) \, v = f(v), \quad \forall  v
\in X \ , 
\label{eq:7-6}
\end{equation}
where
\begin{equation}
a(w, v) = \int_\Omega \nabla w \cdot \nabla v, \quad f(v) = 100
\int_\Omega \sin(2\pi x_{1}) \, \cos(2 \pi x_{2}) \, v,  
\label{eq:7-6a}
\end{equation}
and
\begin{equation}
g(u(\bm \mu), \bm \mu) = \exp(\sin(\mu_2 u(\bm \mu))).  
\label{eq:7-6b}
\end{equation}
The finite element (FE) approximation space is $X = \{v \in H_0^1(\Omega) : v|_K \in \mathcal{P}^3(T), \  \forall T \in \mathcal{T}_h \}$, where $\mathcal{P}^3(T)$ is a space of polynomials of degree $3$ on an element $T \in \mathcal{T}_h$ and $\mathcal{T}_h$ is a finite element grid of $32 \times 32$ quadrilaterals. The dimension of the FE space is $\mathcal{N} = 9409$.




\subsection{Reduced basis approximation}

We introduce the parameter sample set $S_N = \{\bm \mu_1 \in {\cal D} ,\cdots, \bm \mu_N \in {\cal D}\}$, and associated RB space $W_N^u = \mbox{span} \{\zeta_j \equiv  u(\bm \mu_j), \ 1 \leq j \leq N\}$, where
$u(\bm \mu_j)$ is the solution of~(\refeq{eq:7-6}) for $\bm \mu = \bm \mu_j$. We then orthonormalize the $\zeta_j, 1 \leq j \leq N,$ with respect to $(\cdot,\cdot)_X$ so that $(\zeta_i,\zeta_j)_X = \delta_{ij}, 1 \leq i,j \leq N$. The RB approximation is obtained by a standard Galerkin projection: given $\bm \mu \in {\cal D}$, we evaluate $s_{N}(\bm \mu) = \int_{\Omega} u_{N}(\bm\mu)$, where $u_{N}(\bm\mu) \in W_N^u$ is the solution of
\begin{equation}
a(u_{N}(\bm\mu),v) + \mu_1 \int_{\Omega} g(u_{N}(\bm \mu), \bm \mu) v = f(v), \quad
\forall  v \in W_N^u  . 
\label{eq:7-7}
\end{equation} 
We now express $u_N(\bm \mu) = \sum_{n=1}^N \alpha_{N,n} (\bm \mu) \zeta_n$ and choose
test functions $v = \zeta_j, \ 1 \leq j \leq N$, in~(\refeq{eq:7-7}), we obtain the nonlinear algebraic system
\begin{equation}
\bm A_{N} \bm \alpha_N(\bm \mu) + \mu_1\bm G_N(\bm \alpha_N(\bm \mu), \bm \mu) = \bm F_N
\label{eq:7-8}
\end{equation} 
where, for $1 \le n, j \le N$, we have
\begin{equation}
A_{N, jn} = a(\zeta_j,\zeta_n), \quad F_{N,j} = f(\zeta_j), 
\label{eq:7-9}
\end{equation} 
and 
\begin{equation}
G_{N, j}(\bm \alpha_N(\bm \mu), \bm \mu) = \int_{\Omega} g(u_{N}(\bm \mu), \bm \mu) \zeta_j  .
\label{eq:7-10}
\end{equation} 
Both $\bm A_N$ and $\bm F_N$ can be pre-computed in the offline stage owing to their parameter independence, whereas $\bm G_N$ can not be pre-computed due to the nonlinearity of the function $g$. 

We use Newton method to linearize (\ref{eq:7-8}) at a current iterate $\bar{\bm \alpha}_N(\bm \mu)$ to arrive at the following linear system 
\begin{equation}
\left( \bm A_N \ + \mu_1\bm D_N(\bar{\bm \alpha}_N(\bm \mu), \bm \mu)  \right) \delta \bm \alpha_N(\bm \mu) = \bm F_N - \bm A_N \bar{\bm \alpha}_N(\bm \mu) - \mu_1\bm G_N(\bar{\bm \alpha}_N(\bm \mu), \bm \mu) 
\label{eq:7-11}
\end{equation} 
where, for $1 \le n, j \le N$, we have
\begin{equation}
D_{N, j n}(\bar{\bm \alpha}_N(\bm \mu),\bm  \mu) = \int_{\Omega} g'_u(\bar{u}_{N}(\bm \mu),\bm \mu) \zeta_n \zeta_j \ .
\label{eq:7-12}
\end{equation} 
Here $g'_u$ is the partial derivative of $g$ with respect to the first argument. Unfortunately, the matrix $\bm D_N$ can not be pre-computed in the offline stage due to its dependency on $g'_u(\bar{u}_{N}(\bm \mu),\bm \mu)$. Consequently, although the linear system (\ref{eq:7-11}) is small, it is computationally expensive due to the $\mathcal{N}$-dependent complexity of forming both $\bm G_N$ and $\bm D_N$. As a result, the RB approximation does not offer a significant speedup over the FE approximation.

\subsection{Reduced basis approximation via empirical interpolation}

To recover online ${\cal N}$-independence, we simply replace $g(u_N(\bm \mu),\bm \mu)$ in~(\ref{eq:7-7}) with the approximation $g_M(\bm x, \bm \mu) = \sum_{m=1}^M \beta_{M, \, m}(\bm \mu) \psi_m( \bm x)$ based upon the empirical interpolation approach. This requires us to pre-compute the interpolation point set $T_M = \{\widehat{\bm x}_m\}_{m=1}^M$ and the interpolating basis set $W_M^g = \mbox{span} \{\psi_m(\bm x), 1 \le m \le M\}$ for the nonlinear parametrized function $g(u(\bm \mu), \bm \mu)$ in the offline stage. For the original EIM described in Section \ref{section2.2}, we use the sample set $S_N$ and the basis set $W_N^u$ to construct $T_M$ and $W_M^g$. In this case, $M$ can not exceed $N$. For the first-order EIM described in Section \ref{section2.3}, we need to construct the Taylor  space $W_K^{\partial g}$ in (\ref{taylorspace}) using $S_N, W_N^u$ and first-order partial derivatives of $g(\zeta, \bm \mu)$ for $(\zeta, \bm \mu) = (\zeta_n, \bm \mu_n), 1 \le n \le N$. Henceforth, $M$ can be chosen many times greater than $N$. 

In practice, the empirical interpolation is carried out over a set of discretization points on the physical domain $\Omega$. There are two different options for the choice of the discretization points: ({\em i}) nodal points on all elements in the mesh and ({\em ii})  quadrature points on all elements in the mesh. We observe through numerical experiments that the  quadrature points yield more accurate approximation than the nodal points. Hence, the interpolation points are selected from the set of quadrature points on  elements in the mesh. We are going to describe the RB approximation via  empirical interpolation, which turns out to be exactly the same procedure for both the original EIM and the first-order EIM. 

In the online stage, our RB approximation via empirical interpolation is stated as: Given $\bm \mu \in {\cal D}$,  we evaluate $s_{N,M}(\bm \mu) = \int_{\Omega} u_{N,M}(\bm\mu)$, where $u_{N,M}(\bm\mu) \in W_N^u$ is the solution of 
\begin{equation}
a(u_{N,M}(\bm\mu),v) + \mu_1 \sum_{m=1}^M \beta_{M, \, m}(\bm \mu)  \int_{\Omega} \psi_m v = f(v), \quad
\forall  v \in W_N^u  ,
\label{eq:7-7b}
\end{equation} 
where the coefficients $\beta_{M, \, m}(\bm \mu)$ are computed from the following linear system
\begin{equation}
\sum_{m=1}^M  \psi_m(\widehat{\bm x}_k)   \beta_{M,m}(\bm \mu) = g(u_{N,M}(\widehat{\bm x}_k, \bm \mu), \bm \mu), \quad 1 \le k \le M .
\end{equation}
We now express $u_{N,M}(\bm \mu) = \sum_{n=1}^N \alpha_{N,M,n} (\bm \mu) \zeta_n$ and choose
test functions $v = \zeta_j, \ 1 \leq j \leq N$, in~(\refeq{eq:7-7b}), we obtain the nonlinear algebraic system
\begin{equation}
\bm A_N \bm \alpha_{N,M}(\bm \mu) + \mu_1 \bm E_{N,M} \bm g_M (\bm \alpha_{N,M}(\bm \mu), \bm \mu) = \bm F_N, 
\label{eq:7-8b}
\end{equation} 
where $\bm E_{N,M} = \bm C_{N,M} \bm B_M^{-1}$, for $1 \le j \le N$ and $1 \le m, k \le M$, we have
\begin{equation}
C_{N, M, j m} = \int_{\Omega} \psi_m \zeta_j, \quad B_{M, km} =  \psi_m(\widehat{\bm x}_k), 
\label{eq:7-10b}
\end{equation} 
and
\begin{equation}
g_{M, k} (\bm \alpha_{N,M}(\bm \mu),\bm \mu) = g\left(  \sum_{n=1}^N \alpha_{N,M,n} (\bm \mu) \zeta_n(\widehat{\bm x}_k), \bm \mu\right) .
\label{eq:7-10c}
\end{equation} 
The vector $\bm F_N$ and  matrices $\bm A_N, \bm E_{N,M}$ can be pre-computed in the offline stage since they are independent of $\bm \mu$. \revise{We form and store the Jacobian matrix of the FOM to  compute $\bm A_N$ using Jacobian-vector products during the offline stage. For large-scale problems, forming and storing the Jacobian matrix can be costly in terms of both memory storage and computational time. One possible remedy to reduce the memory storage and computational time is to use a Jacobian-free technique to compute Jacobian-vector products via finite difference. However,  finite difference does not yield the exact Jacobian-vector products, a more rigorous method is to use automatic differentiation to compute the Jacobian-vector products exactly  \cite{Vila-Perez2022}.}
 
We use Newton method to linearize (\ref{eq:7-8b}) at a given iterate $\bar{\bm \alpha}_{N,M}(\bm \mu)$ to arrive at the following linear system 
\begin{multline}    
\left( \bm A_N \ + \mu_1\bm E_{N,M} \bm H_{M,N}(\bar{\bm \alpha}_{N,M}(\bm \mu), \bm \mu)  \right) \delta \bm \alpha_N(\bm \mu) =  \\ \bm F_N - \bm A_N \bar{\bm \alpha}_{N,M}(\bm \mu) - \mu_1 \bm E_{N,M} \bm g_M (\bar{\bm \alpha}_{N,M}(\bm \mu), \bm \mu) 
\label{eq:7-11b}
\end{multline}
where, for $1 \le i \le N$ and $1 \le k \le M$, we have
\begin{equation}
H_{M, N, k i}(\bar{\bm \alpha}_N(\bm \mu), \bm \mu) = g'_u\left(  \sum_{n=1}^N \bar{\alpha}_{N,M,n} (\bm \mu) \zeta_n(\widehat{\bm x}_k), \bm \mu\right) \zeta_i(\widehat{\bm x}_k) .
\label{eq:7-12b}
\end{equation} 
\revise{In the online stage,  at each Newton iteration, we compute $\bm g_M (\bar{\bm \alpha}_{N,M}(\bm \mu), \bm \mu)$ from (\ref{eq:7-10c}) and $\bm H_{M,N}(\bar{\bm \alpha}_{N,M}(\bm \mu), \bm \mu)$ from (\ref{eq:7-12b}), and solve the linear system (\ref{eq:7-11b}). The conline omplexity of evaluating $\bm g_M (\bar{\bm \alpha}_{N,M}(\bm \mu), \bm \mu)$ is $O(MN)$, while that of evaluating $\bm H_{M,N}(\bar{\bm \alpha}_{N,M}(\bm \mu), \bm \mu)$ is $O(MN^2)$. Hence, the overall complexity of solving the linear system (\ref{eq:7-11b}) per Newton iteration is $O(MN^2 + N^3)$. Since $M \ge N$, the computational complexity per Newton iteration becomes $O(MN^2)$. As a result, the RB approximation via empirical interpolation can be  orders of magnitude  faster than the RB approximation described earlier.}

It is important to note that the computational complexity scales linearly with $M$.  Hence, it can be advantageous to increase $M$ to improve the accuracy of the RB approximation via empirical interpolation. As $M$ increases, we expect that the RB approximation via empirical interpolation will converge to the standard RB approximation. To assess the accuracy of the output approximation, we define the following output errors  
\begin{equation}
\epsilon^s_N(\bm \mu) = |s(\bm \mu) - s_N(\bm \mu)|, \quad \epsilon^s_{N,M}(\bm \mu) = |s(\bm \mu) - s_{N,M}(\bm \mu)| 
\end{equation}
for the standard RB approximation and the RB approximation via empirical interpolation, respectively. Similarly, we introduce the following errors to assess the approximation of the solution
\begin{equation}
\epsilon^u_N(\bm \mu) = \|u(\bm \mu) - u_N(\bm \mu)\|_{X}, \quad \epsilon^u_{N,M}(\bm \mu) = \|u(\bm \mu) - u_{N,M}(\bm \mu)\|_X . 
\end{equation}
In general, we expect $\epsilon^s_{N,M}(\bm \mu) \ge\epsilon^s_N(\bm \mu)$ and $\epsilon^u_{N,M}(\bm \mu) \ge\epsilon^u_N(\bm \mu)$. The effectivities as defined below
\begin{equation}
\eta_{N,M}^s(\bm \mu) = \frac{\epsilon^s_{N,M}(\bm \mu)}{\epsilon^s_N(\bm \mu)}, \quad \eta_{N,M}^u(\bm \mu) = \frac{\epsilon^u_{N,M}(\bm \mu)}{\epsilon^u_N(\bm \mu)}
\end{equation}
will measure the accuracy of the RB approximation via empirical interpolation relative to the standard RB approximation. If the effectivities are close to unity, the RB approximation via empirical interpolation can be considered as accurate as the standard RB approximation. However, if they are much greater than unity, the RB approximation via empirical interpolation will be not accurate enough.

\subsection{Numerical results}
\label{section3.4}

We present numerical results for the model problem of Section \ref{section3.1}. \revise{The first-order EIM results reported herein are obtained by using the FOEIM Algoeirhm I.} We introduce a uniform grid of size $N_{\rm
  Test} = 30 \times 30$ as a parameter test sample $S^g_{\rm Test}$, and define 
\begin{equation}
\bar{\epsilon}_{N}^s = \frac{\sum_{\bm \mu \in S^g_{\rm Test}}\epsilon_N^s (\bm \mu)}{ \sum_{\bm \mu \in S^g_{\rm Test}} |s (\bm \mu)|} , \quad  \bar{\epsilon}_{N}^u =  \frac{\sum_{\bm \mu \in S^g_{\rm Test}} \epsilon_N^u (\bm \mu)}{ \sum_{\bm \mu \in S^g_{\rm Test}} \|u(\bm \mu)\|_X }.  
\end{equation}  
The quantities $\bar{\epsilon}_{N,M}^s$ and $\bar{\epsilon}_{N,M}^u$ are similarly defined via $\epsilon_{N,M}^s (\bm \mu)$ and $\epsilon_{N,M}^u (\bm \mu)$, respectively.  We present in Figure \ref{fig3} $\bar{\epsilon}_{N}^s$, $\bar{\epsilon}_{N,M}^s$,  $\bar{\epsilon}_{N}^u$, and $\bar{\epsilon}_{N,M}^u$ as a function of $N$.  As $M$ increases, $\bar{\epsilon}_{N,M}^s$ (respectively, $\bar{\epsilon}_{N,M}^u$) converges to $\bar{\epsilon}_{N}^s$ (respectively, $\bar{\epsilon}_{N}^u$). The RB approximation based on the first-order EIM with $M=8N$ is almost as accurate as the standard RB approximation. However, the RB approximation based on the original EIM is significantly less accurate than the standard RB approximation.

\begin{figure}[hthbp]
	\centering
	\begin{subfigure}[b]{0.49\textwidth}
		\centering
		\includegraphics[width=\textwidth]{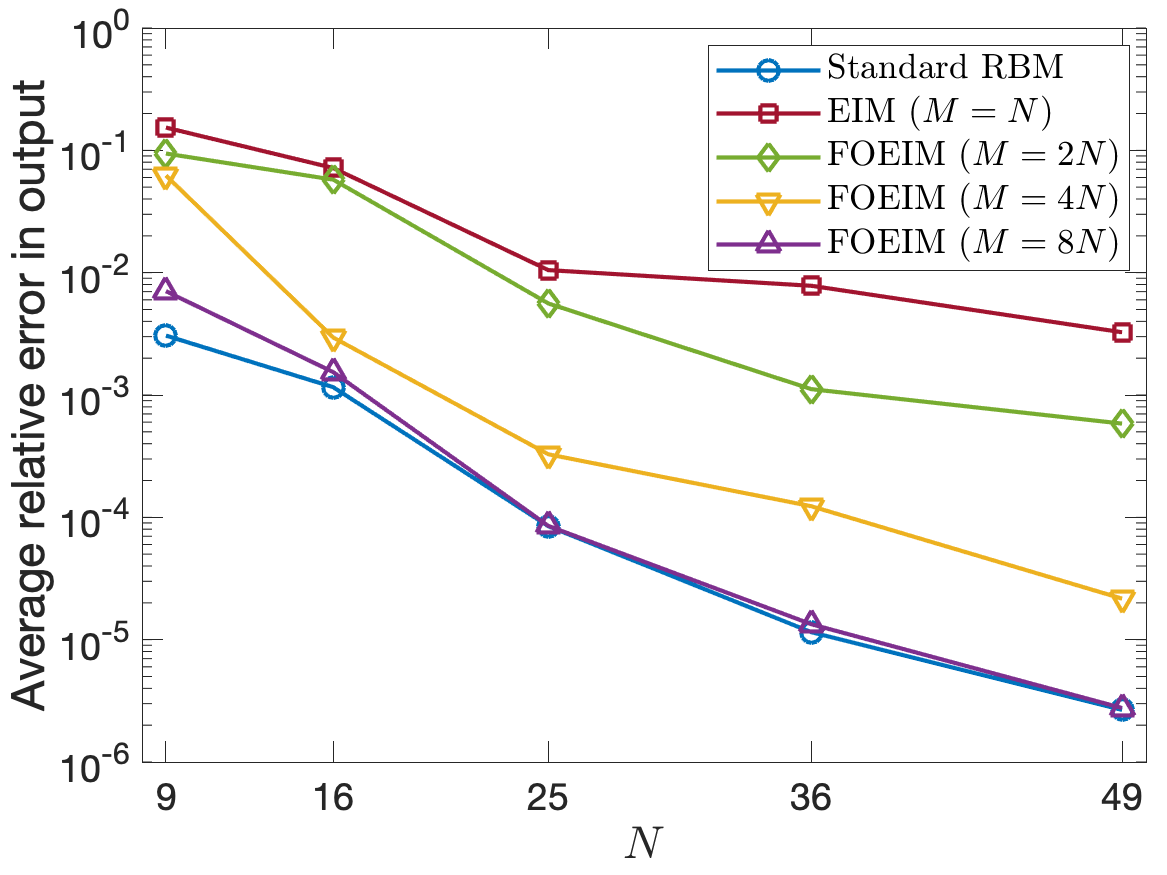}
		\caption{$\bar{\epsilon}_{N}^s$ and $\bar{\epsilon}_{N,M}^s$}
	\end{subfigure}
	\hfill
	\begin{subfigure}[b]{0.49\textwidth}
		\centering
		\includegraphics[width=\textwidth]{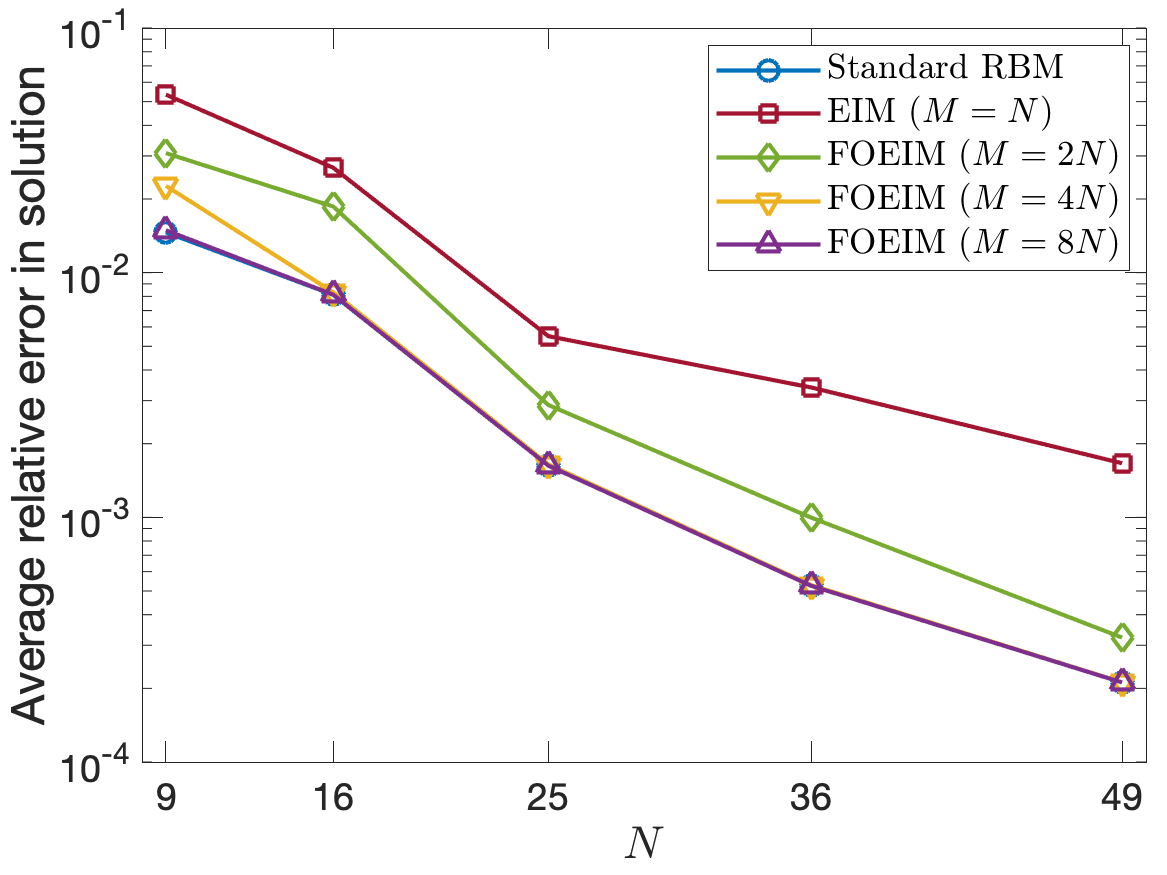}
		\caption{$\bar{\epsilon}_{N}^u$ and $\bar{\epsilon}_{N,M}^u$}
	\end{subfigure}
	\caption{Convergence of the average relative error in output (a) and solution (b) for the model problem of Section \ref{section3.1}.}
	\label{fig3}
\end{figure}
 
To compare the original EIM and the first-order EIM, we  introduce 
\begin{equation}
\bar{\eta}_{N,M}^s = \frac{1}{N_{\rm Test}}\sum_{\bm \mu \in S^g_{\rm Test}} \eta_{N,M}^s (\bm \mu), \quad   \bar{\eta}_{N,M}^u = \frac{1}{N_{\rm Test}} \sum_{\bm \mu \in S^g_{\rm Test}} \eta_{N,M}^u (\bm \mu) .
\end{equation}
We display in Table~\ref{tab1} 
$\bar{\eta}_{N,M}^s$ and $\bar{\eta}_{N,M}^u$ as a function of $N$. We observe that the average effectivities decrease toward unity as $M$ increases. Furthermore, the first-order EIM with $M=8N$ yields significantly smaller output effectivties by several orders of magnitudes than the original EIM for the same dimension $N$. It is interesting to point out that the output effectivties are considerably larger than the solution effectivities. 

\begin{table}[htbp]
\centering
\small
	\begin{tabular}{|c|cc|cc|cc|cc|cc|}
		\cline{1-9}
    &		 
	 \multicolumn{2}{|c|}{$M=N$} & \multicolumn{2}{c|}{$M=2N$} & 
		 \multicolumn{2}{c|}{$M=4N$} &
		 \multicolumn{2}{c|}{$M=8N$} \\   
   $N$ & $\bar{\eta}_{N,M}^s$ & $\bar{\eta}_{N,M}^u$ & $\bar{\eta}_{N,M}^s$ & $\bar{\eta}_{N,M}^u$ & $\bar{\eta}_{N,M}^s$ & $\bar{\eta}_{N,M}^u$ & $\bar{\eta}_{N,M}^s$ & $\bar{\eta}_{N,M}^u$ \\
		\cline{1-9}
9  &  193.38  &  4.59  &  102.45  &  2.17  &  26.28  &  1.41  &  5.73  &  1.03  \\  
 16  &  168.25  &  3.03  &  130.03  &  2.07  &  10.81  &  1.04  &  2.34  &  1.00  \\  
 25  &  396.44  &  3.82  &  224.32  &  2.14  &  12.31  &  1.01  &  2.83  &  1.00  \\  
 36  &  1223.7  &  6.87  &  190.39  &  2.06  &  17.51  &  1.01  &  1.95  &  1.00  \\  
 49  &  1310.14  &  7.67  &  294.96  &  1.52  &  17.92  &  1.00  &  1.98  &  1.00  \\  
		\hline
	\end{tabular}
	\caption{Average effectivities for the model problem of Section \ref{section3.1}. The column with $M=N$ corresponds to the original EIM, while the columns with $M >N$ correspond  to the first-order EIM.} 
	\label{tab1}
\end{table}

We present in Table~\ref{tab2} the online
computational times to calculate $s_{N}(\bm \mu)$ and $s_{N,M}(\bm \mu)$ as a function of $N$. The values are normalized with respect to the computational time  of the truth approximation output $s(\bm \mu)$. The computational saving is significant: for an relative accuracy of about 0.0001 ($N = 25$, $M = 200$) in the output, the reduction in online cost is more than a factor of 1000; this is mainly because the matrix assembly of the nonlinear terms for the truth approximation is computationally very expensive. The standard RB approximation has similar computational times as the truth FE approximation, and is between 100 and 1000 times slower than the RB approximation via empirical interpolation. \revise{We notice that using $M = N$ often requires more Newton iterations to converge than using $M > N$ especially when $N$ is relatively small. As a result, the online computational time with $M=N$ is slightly higher than that $M>N$ especially for $N = 9$ and $N=16$.}

\begin{table}[htbp]
\centering
\small
	\begin{tabular}{|c|c|c|c|c|c|c|}
		\cline{1-7}
  & FEM  & RBM & EIM & FOEIM  & FOEIM & FOEIM \\  
   $N$ & $s(\bm \mu)$ &  $s_N(\bm \mu)$ & $M=N$ & $M=2N$ & $M=4N$ & $M=8N$ \\
		\cline{1-7}
9 & 1 &  5.37\mbox{e-}1  &  1.04\mbox{e-}3  &  4.93\mbox{e-}4  &  3.31\mbox{e-}4  &  3.46\mbox{e-}4  \\  
 16  & 1 &  5.44\mbox{e-}1  &  4.67\mbox{e-}4  &  4.59\mbox{e-}4  &  4.51\mbox{e-}4  &  5.73\mbox{e-}4  \\  
 25 & 1  &  5.61\mbox{e-}1  &  5.17\mbox{e-}4  &  5.19\mbox{e-}4  &  5.98\mbox{e-}4  &  9.65\mbox{e-}4  \\  
 36 & 1  &  1.51\mbox{e-}0  &  2.49\mbox{e-}3  &  3.43\mbox{e-}3  &  3.93\mbox{e-}3  &  5.00\mbox{e-}3  \\  
 49  & 1  &  2.07\mbox{e-}0  &  3.48\mbox{e-}3  &  4.15\mbox{e-}3  &  5.09\mbox{e-}3  &  6.77\mbox{e-}3  \\   
		\hline
	\end{tabular}
	\caption{Online computational times (normalized with respect to the time to
  solve for $s(\bm \mu)$) for the model problem of Section \ref{section3.1}.} 
	\label{tab2}
\end{table}

\section{Nonlinear diffusion equations}

\subsection{A model problem}
\label{section4.1}

We consider a parametrized nonlinear heat conduction problem 
\begin{equation}
-\nabla \cdot \left(\kappa(u, \bm \mu) \nabla u \right)  = 0 \quad \mbox{in } \Omega, 
\end{equation} 
 with homogeneous boundary conditions
 \begin{equation}
u = 0 \quad \mbox{on } \Gamma_{\rm D}, \quad  \kappa(u, \bm \mu)  \nabla u \cdot \bm n = 0 \quad \mbox{on } \Gamma_{\rm N}, 
\end{equation} 
and non-homogeneous Neumann condition
 \begin{equation}
 \kappa(u, \bm \mu)  \nabla u \cdot \bm n = \mu_1 \quad \mbox{on } \Gamma_{\rm Q} .
\end{equation} 
Here $\Omega$ is the $T$-shaped domain as shown in Figure \ref{fig4}(a). $\Gamma_{\rm D}$ is the top boundary and $\Gamma_{\rm Q}$ is the bottom boundary, while $\Gamma_{\rm N}$ is the remaining part of the  boundary. The parameter domain is ${\cal D} \equiv [1, 10] \times [0, 10]$. The thermal conductivity is a Gaussian function of the form
\begin{equation}
\kappa(u, \bm \mu)  = \exp( - (u - \mu_2)^2 )   .
\end{equation}
The output of interest is the average of the field variable over the
physical domain. For any given $\bm \mu \in
{\cal D}$, we evaluate $s(\bm \mu) = \int_\Omega u(\bm \mu)$, where  $u(\bm \mu) \in X \subset H_0^1(\Omega) \equiv \{v \in
H^1(\Omega) \mbox{ } | \mbox{ } v|_{\Gamma_{\rm D}} = 0\}$ is the solution of
\begin{equation}
\int_\Omega \kappa(u(\bm \mu), \bm \mu) \nabla u \cdot \nabla v = \mu_1 \int_{\Gamma_{\rm Q}} v, \quad \forall  v
\in X \ .  
\label{eq:8-6}
\end{equation}
The FE approximation space is $X = \{v \in H_0^1(\Omega) : v|_K \in \mathcal{P}^3(T), \  \forall T \in \mathcal{T}_h \}$, where $\mathcal{P}^3(T)$ is a space of polynomials of degree $3$ on an element $T \in \mathcal{T}_h$ and $\mathcal{T}_h$ is a mesh of $900$ quadrilaterals. The dimension of $X$ is $\mathcal{N} = 8401$. Figure \ref{fig4} shows  FE solutions at two different parameter points.

\begin{figure}[hthbp]
	\centering
	\begin{subfigure}[b]{0.29\textwidth}
		\centering
		\includegraphics[width=\textwidth]{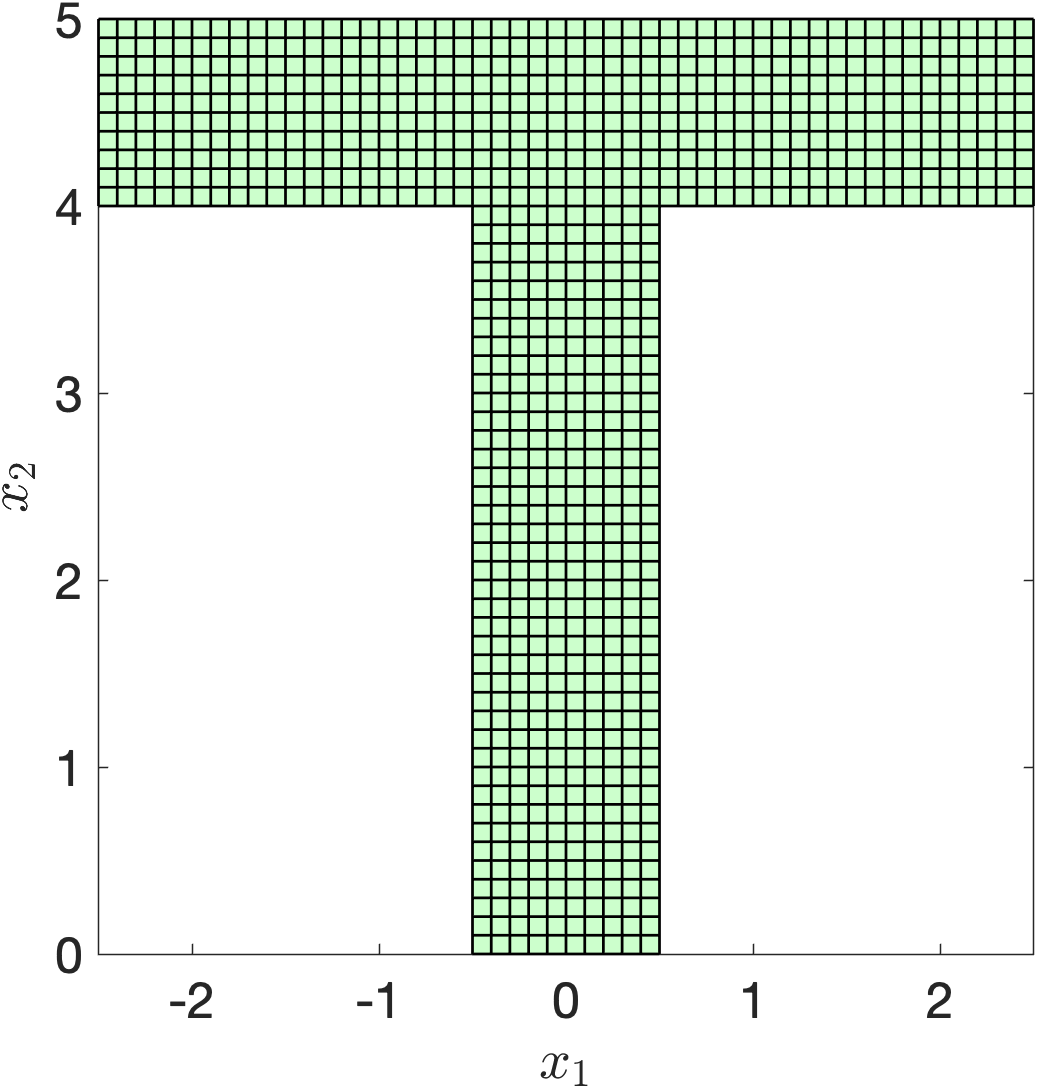}
		\caption{FE mesh}
	\end{subfigure}
	\hfill
	\begin{subfigure}[b]{0.33\textwidth}
		\centering
		\includegraphics[width=\textwidth]{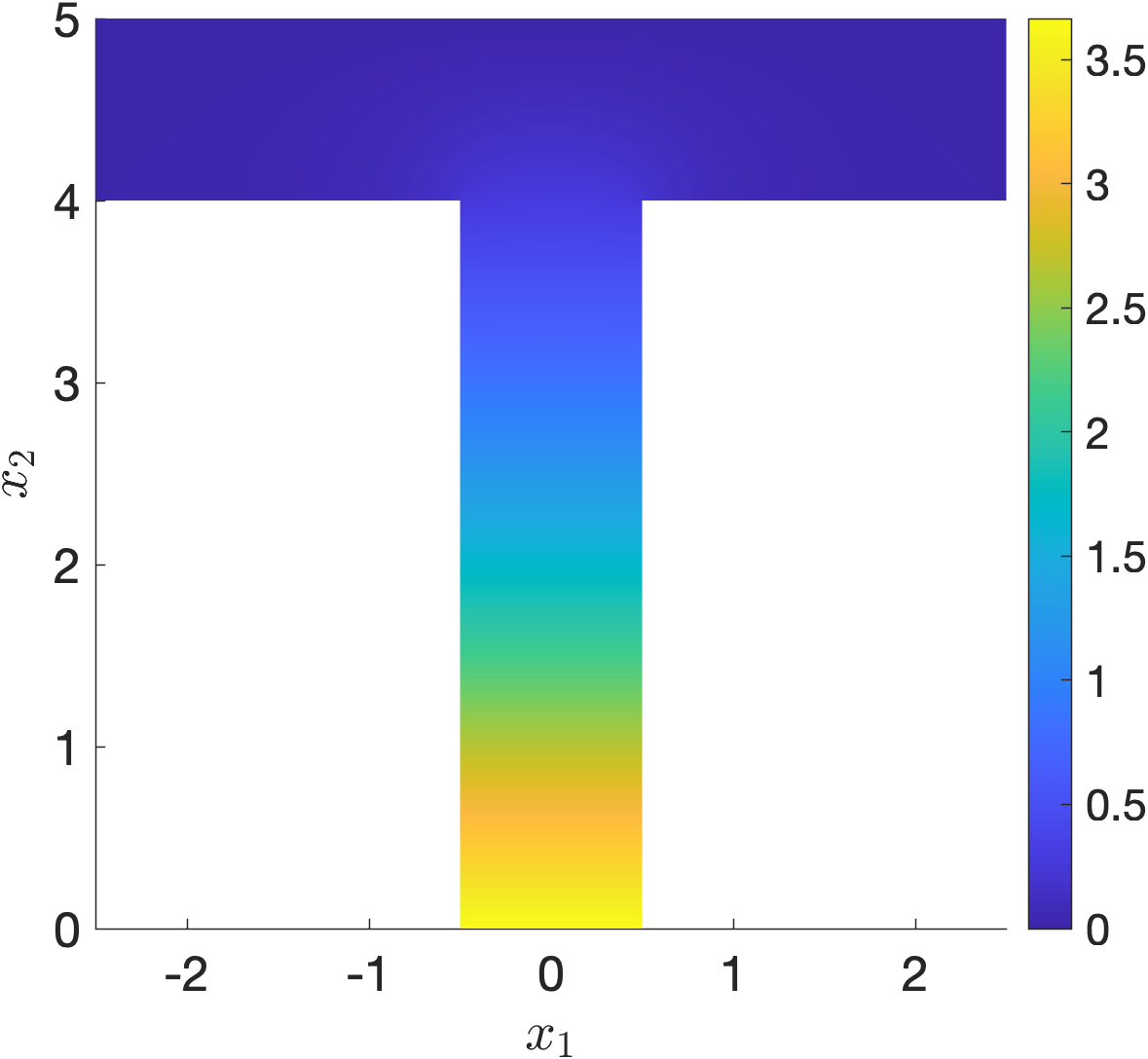}
		\caption{FE solution for $\bm \mu = (1,0)$}
	\end{subfigure}
        \hfill
	\begin{subfigure}[b]{0.323\textwidth}
		\centering
		\includegraphics[width=\textwidth]{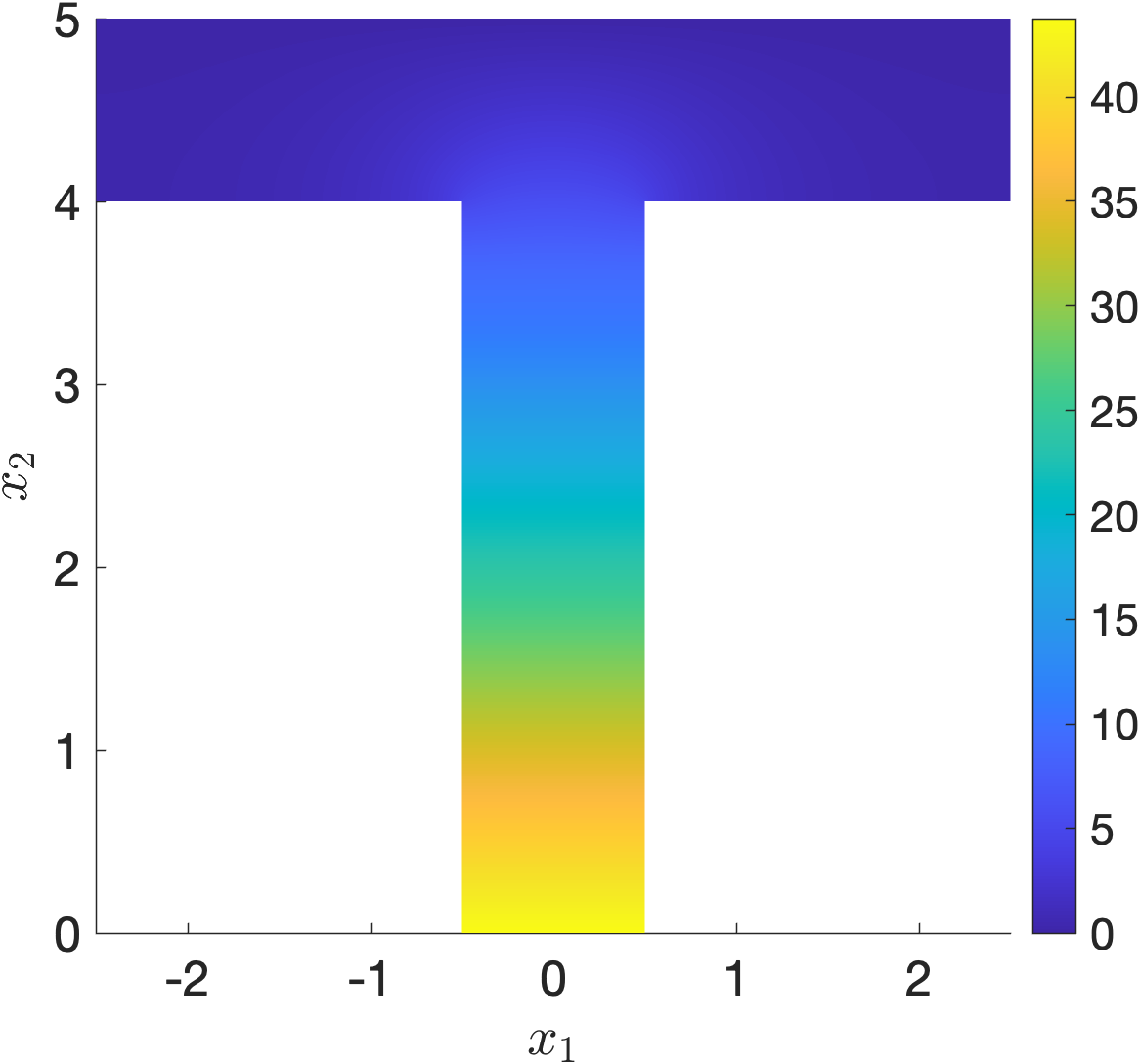}
		\caption{FE solution for $\bm \mu = (10,10)$}
	\end{subfigure}
	\caption{FE mesh and numerical solutions at two parameter points for the nonlinear heat conduction problem.}
	\label{fig4}
\end{figure}
 
\subsection{Reduced basis approximation}

The RB approximation is obtained by a standard Galerkin projection: given $\bm \mu \in {\cal D}$, we evaluate $s_{N}(\bm \mu) = \int_{\Omega} u_{N}(\bm\mu)$, where $u_{N}(\bm\mu) \in W_N^u$ is the solution of
\begin{equation}
\int_\Omega \kappa(u_N(\bm \mu),\bm \mu) \nabla u_N \cdot \nabla v = \mu_1 \int_{\Gamma_{\rm Q}} v, \quad
\forall  v \in W_N^u  . 
\label{eq:8-7}
\end{equation} 
We now express $u_N(\bm \mu) = \sum_{n=1}^N \alpha_{N,n} (\bm \mu) \zeta_n$ and choose
test functions $v = \zeta_j, \ 1 \leq j \leq N$, in~(\refeq{eq:8-7}), we obtain the nonlinear algebraic system
\begin{equation} 
\bm A_{N}(\bm \alpha_N(\bm \mu), \bm \mu) \, \bm \alpha_N(\bm \mu)  = \mu_1 \bm F_N,
\label{eq:8-8}
\end{equation} 
where, for $1 \le n, j \le N$, we have
\begin{equation}
A_{N, jn}(\bm \alpha_N(\bm \mu), \bm \mu) =  \int_\Omega \kappa(u_N(\bm \mu), \bm \mu) \nabla \zeta_n \cdot \nabla  \zeta_j, \quad F_{N,j} = \int_{\Gamma_{\rm Q}} \zeta_j . 
\label{eq:8-9}
\end{equation} 
While $\bm F_N$ can be pre-computed in the offline stage, $\bm A_N$ can not be pre-computed due to the nonlinearity of the function $\kappa$.  Although the nonlinear system (\ref{eq:8-8}) has a small number of unknowns, it is computationally expensive due to the $\mathcal{N}$-dependent complexity of forming $\bm A_N$. As a result, the RB approximation does not offer a significant speedup over the FE approximation.

\subsection{Reduced basis approximation via empirical interpolation}

For this particular nonlinear PDE, we  apply the empirical interpolation to $\kappa(u(\bm \mu), \bm \mu) \nabla u(\bm \mu)$, which is a vector-valued function. To this end, we introduce 
\begin{equation}
g(u(\bm \mu), \bm \mu) =  \kappa(u(\bm \mu), \bm \mu) \frac{\partial u(\bm \mu)}{\partial x_1}, \quad h(u(\bm \mu), \bm \mu) =  \kappa(u(\bm \mu), \bm \mu) \frac{\partial u(\bm \mu)}{\partial x_2} .
\label{eq:9-0}
\end{equation}
The RB approximation (\ref{eq:8-7}) is equivalent to finding $u_{N}(\bm\mu) \in W_N^u$ such that
\begin{equation}
\int_\Omega g(u_N(\bm \mu), \bm \mu) \frac{\partial v}{\partial x_1}  + \int_\Omega h(u_N(\bm \mu), \bm \mu) \frac{\partial v}{\partial x_2}  = \mu_1 \int_{\Gamma_{\rm Q}} v, \quad
\forall  v \in W_N^u  . 
\label{eq:9-7}
\end{equation} 
Since the two nonlinear functions in  (\ref{eq:9-0}) depends not only on the field variable but also its spatial gradient, their evaluation requires us to compute the gradient of the field variable.

To develop an efficient RB approximation, we replace the two nonlinear functions, namely $g(u_N(\bm \mu), \bm \mu)$ and $h(u_N(\bm \mu), \bm \mu)$, in~(\ref{eq:9-7}) with 
\begin{equation}
 g_M(\bm x, \bm \mu) = \sum_{m=1}^M \beta_{M, \, m}(\bm \mu) \psi_m^g( \bm x) , \qquad h_M(\bm x, \bm \mu) = \sum_{m=1}^M \gamma_{M, \, m}(\bm \mu) \psi^h_m( \bm x)     
\end{equation}
to obtain $u_{N,M}(\bm \mu) \in W_N^u$ as the solution of
\begin{equation}
\sum_{m=1}^M \beta_{M, \, m}(\bm \mu)  \int_\Omega \psi_m^g \frac{\partial v}{\partial x_1}  +  \sum_{m=1}^M \gamma_{M, \, m}(\bm \mu) \int_\Omega \psi^h_m  \frac{\partial v}{\partial x_2}  = \mu_1 \int_{\Gamma_{\rm Q}} v, \quad
\forall  v \in W_N^u  .
\label{eq:9-8}
\end{equation} 
Here  $\beta_{M, \, m}(\bm \mu)$ and  $\gamma_{M, \, m}(\bm \mu)$ are computed from 
\begin{equation}
\begin{split}
\sum_{m=1}^M  \psi^g_m(\widehat{\bm x}^g_k)   \beta_{M,m}(\bm \mu) = g(u_{N,M}(\widehat{\bm x}_k^g, \bm \mu), \bm \mu), \quad 1 \le k \le M , \\
\sum_{m=1}^M  \psi^h_m(\widehat{\bm x}^h_k)   \gamma_{M,m}(\bm \mu) = h(u_{N,M}(\widehat{\bm x}_k^h, \bm \mu), \bm \mu), \quad 1 \le k \le M ,
\end{split}
\end{equation}
where $T_M^g = \{\widehat{\bm x}^g_m\}_{m=1}^M$, $W_M^g = \mbox{span}\{\psi_m^g(\bm x) \}_{m=1}^M$, and $T_M^h = \{\widehat{\bm x}^h_m\}_{m=1}^M$, $W_M^h = \mbox{span}\{\psi_m^h(\bm x) \}_{m=1}^M$ are the interpolation point sets and basis sets for approximating $g(u_N(\bm \mu), \bm \mu)$ and $h(u_N(\bm \mu), \bm \mu)$, respectively.  These interpolation point sets and basis sets are pre-computed in the offline stage. For the first-order EIM described in Section \ref{section2.3}, we employ the partial derivatives of the nonlinear functions  as follows 
\begin{multline}    
\label{eqrho160}
\vartheta^g_{(n-1)N + k}(\bm x) =  \frac{\partial \kappa(\zeta_n(\bm x), \bm \mu_n)} {\partial u} \frac{\partial \zeta_n(\bm \mu_n)}{\partial x_1} (\zeta_k(\bm x) - \zeta_n(\bm x)) \ + \\   \kappa(\zeta_n(\bm x), \bm \mu_n) \frac{\partial (\zeta_k(\bm x) - \zeta_n(\bm x))}{\partial x_1}, 
\end{multline}
and
\begin{equation}
\label{eqrho170}
\vartheta^g_{N^2 + (n-1)N + k}(\bm x) = \frac{\partial \kappa(\zeta_n(\bm x), \bm \mu_n)}{\partial \bm \mu} \frac{\partial \zeta_n(\bm \mu_n)}{\partial x_1}  \cdot (\bm \mu_k - \bm \mu_n),     
\end{equation}
for $1 \le k, n \le N$. Note that $\vartheta^h_{(n-1)N + k}(\bm x)$ and $\vartheta^h_{N^2 + (n-1)N + k}(\bm x)$ are similarly computed.



Next, we express $u_{N,M}(\bm \mu) = \sum_{n=1}^N \alpha_{N,M,n} (\bm \mu) \zeta_n$ and choose
test functions $v = \zeta_j, \ 1 \leq j \leq N$, in~(\refeq{eq:9-8}), we obtain the nonlinear algebraic system
\begin{equation}
 \bm E^g_{N,M} \bm g_M (\bm \alpha_{N,M}(\bm \mu), \bm \mu)  +  \bm E_{N,M}^h \bm h_M (\bm \alpha_{N,M}(\bm \mu), \bm \mu)  = \mu_1 \bm F_N, 
\label{eq:9-8b}
\end{equation} 
where $\bm E^g_{N,M} = \bm C^g_{N,M} [\bm B_M^g]^{-1}$ and $\bm E^h_{N,M} = \bm C^h_{N,M} [\bm 
 B_M^h]^{-1}$, for $1 \le j \le N$ and $1 \le m, k \le M$, we have
\begin{equation}
C^g_{N, M, j m} = \int_{\Omega} \psi^g_m \frac{\partial \zeta_j}{\partial x_1}, \quad C^h_{N, M, j m} = \int_{\Omega} \psi^h_m \frac{\partial \zeta_j}{\partial x_2} 
\label{eq:9-10b}
\end{equation} 
and
\begin{equation}
\begin{split}    
g_{M, k} (\bm \alpha_{N,M}(\bm \mu), \bm \mu) = g\left(  \sum_{n=1}^N \alpha_{N,M,n} (\bm \mu) \zeta_n(\widehat{\bm x}_k^g),  \bm \mu\right), \\
h_{M, k} (\bm \alpha_{N,M}(\bm \mu), \bm \mu) = h\left(  \sum_{n=1}^N \alpha_{N,M,n} (\bm \mu) \zeta_n(\widehat{\bm x}_k^h),  \bm \mu\right) .
\end{split}
\label{eq:9-10c}
\end{equation} 
Note that $\bm E^g_{N,M}$, $\bm E^h_{N,M}$, and $\bm F_{N}$ can be pre-computed in the offline stage since they are independent of $\bm \mu$. 

We use Newton method to linearize (\ref{eq:9-8b}) at a given iterate $\bar{\bm \alpha}_{N,M}(\bm \mu)$ to arrive at the following linear system 
\begin{multline}    
\left(\bm E^g_{N,M} \bm H^g_{M,N}(\bar{\bm \alpha}_{N,M}(\bm \mu), \bm \mu)  + \bm E^h_{N,M} \bm H^h_{M,N}(\bar{\bm \alpha}_{N,M}(\bm \mu), \bm \mu)  \right) \delta \bm \alpha_{N,M}(\bm \mu) =  \\\mu_1 \bm F_N  - \bm E^g_{N,M} \bm g_M (\bar{\bm \alpha}_{N,M}(\bm \mu), \bm \mu) - \bm E^h_{N,M} \bm h_M (\bar{\bm \alpha}_{N,M}(\bm \mu), \bm \mu) 
\label{eq:9-11b}
\end{multline}
where, for $1 \le i \le N$ and $1 \le k \le M$, we have
\begin{equation}
\begin{split}
H^g_{M, N, k i}(\bar{\bm \alpha}_{N,M}(\bm \mu), \bm \mu) = g'_u\left(  \sum_{n=1}^N \bar{\alpha}_{N,M,n} (\bm \mu) \zeta_n(\widehat{\bm x}_k),  \bm \mu\right) \zeta_i(\widehat{\bm x}_k) , \\
H^h_{M, N, k i}(\bar{\bm \alpha}_{N,M}(\bm \mu), \bm \mu) = h'_u\left(  \sum_{n=1}^N \bar{\alpha}_{N,M,n} (\bm \mu) \zeta_n(\widehat{\bm x}_k),  \bm \mu\right) \zeta_i(\widehat{\bm x}_k) .    
\end{split}
\label{eq:9-12b}
\end{equation} 
Once the Newton iteration converges, we  evaluate the  RB output as 
\begin{equation}
s_{N,M}(\bm \mu) = \sum_{n=1}^N  \alpha_{N,M, n}(\bm \mu) L_{N,n}       
\end{equation}
where $L_{N,n} = \int_{\Omega} \zeta_n$ are pre-computed in the offline stage. The complexity of solving the linear system (\ref{eq:9-11b}) per Newton iteration is $O(MN^2)$. Therefore, the RB approximation via empirical interpolation is efficient.



\subsection{Numerical results}

We present numerical results for the model problem of Section \ref{section4.1}. \revise{The first-order EIM results reported herein are obtained by using the FOEIM Algoeirhm I.} 
The parameter test sample $S^g_{\rm Test}$ is a uniform grid of size $N_{\rm
  Test} = 30 \times 30$.  We present in Figure \ref{fig6} $\bar{\epsilon}_{N}^s$, $\bar{\epsilon}_{N,M}^s$,  $\bar{\epsilon}_{N}^u$, and $\bar{\epsilon}_{N,M}^u$ as a function of $N$ and in Table~\ref{tab3} 
$\bar{\eta}_{N,M}^s$ and $\bar{\eta}_{N,M}^u$ as a function of $N$. These quantities were defined in Section \ref{section3.4}. As $M$ increases, $\bar{\epsilon}_{N,M}^s$ (respectively, $\bar{\epsilon}_{N,M}^u$) converges to $\bar{\epsilon}_{N}^s$ (respectively, $\bar{\epsilon}_{N}^u$). The RB approximation based on the first-order EIM with $M=3N$ is almost as accurate as the standard RB approximation, whereas the RB approximation based on the original EIM is significantly less accurate than the standard RB approximation.   We observe that the average effectivities decrease toward unity as $M$ increases. Furthermore, the first-order EIM with $M=3N$ yields significantly smaller output effectivties than the original EIM for the same dimension $N$.  Hence, the first-order EIM provides more accurate reduced basis approximation than the original EIM.

\begin{figure}[h]
	\centering
	\begin{subfigure}[b]{0.49\textwidth}
		\centering
		\includegraphics[width=\textwidth]{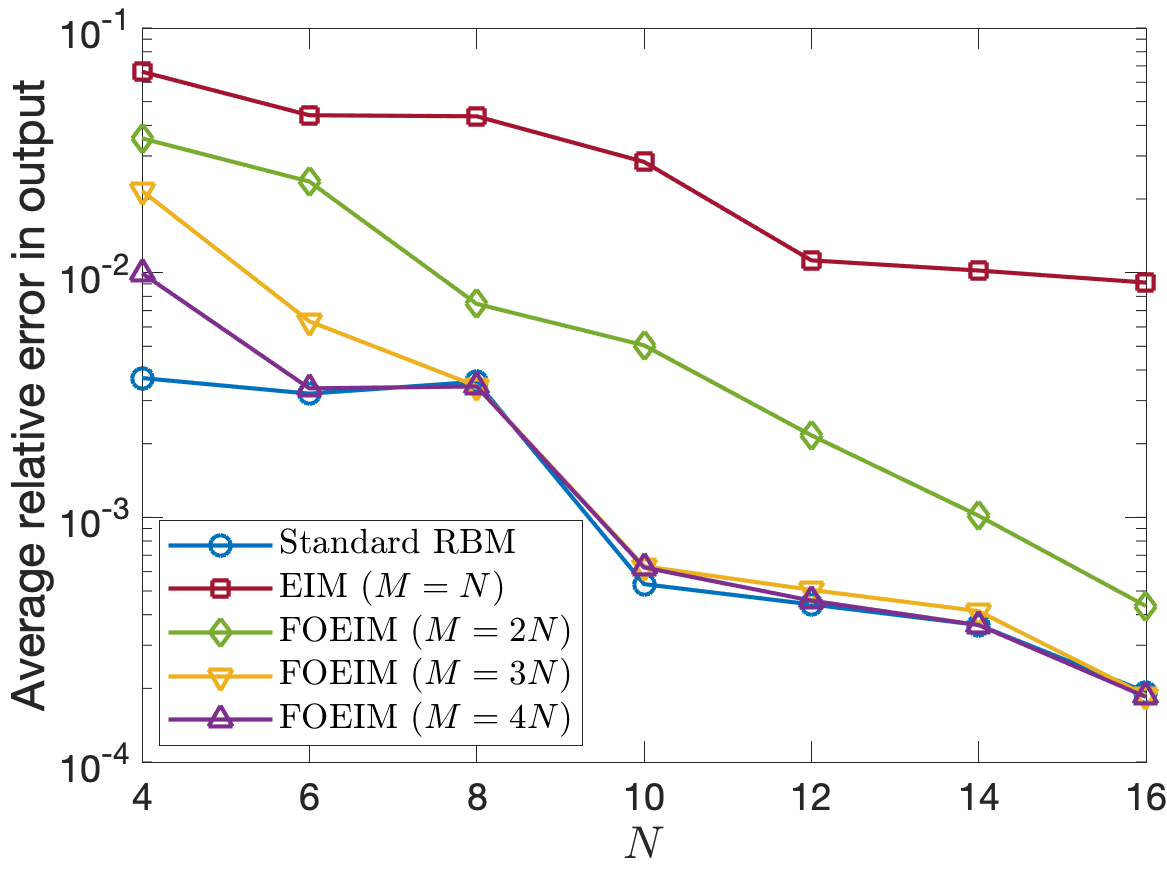}
		\caption{$\bar{\epsilon}_{N}^s$ and $\bar{\epsilon}_{N,M}^s$}
	\end{subfigure}
	\hfill
	\begin{subfigure}[b]{0.49\textwidth}
		\centering		\includegraphics[width=\textwidth]{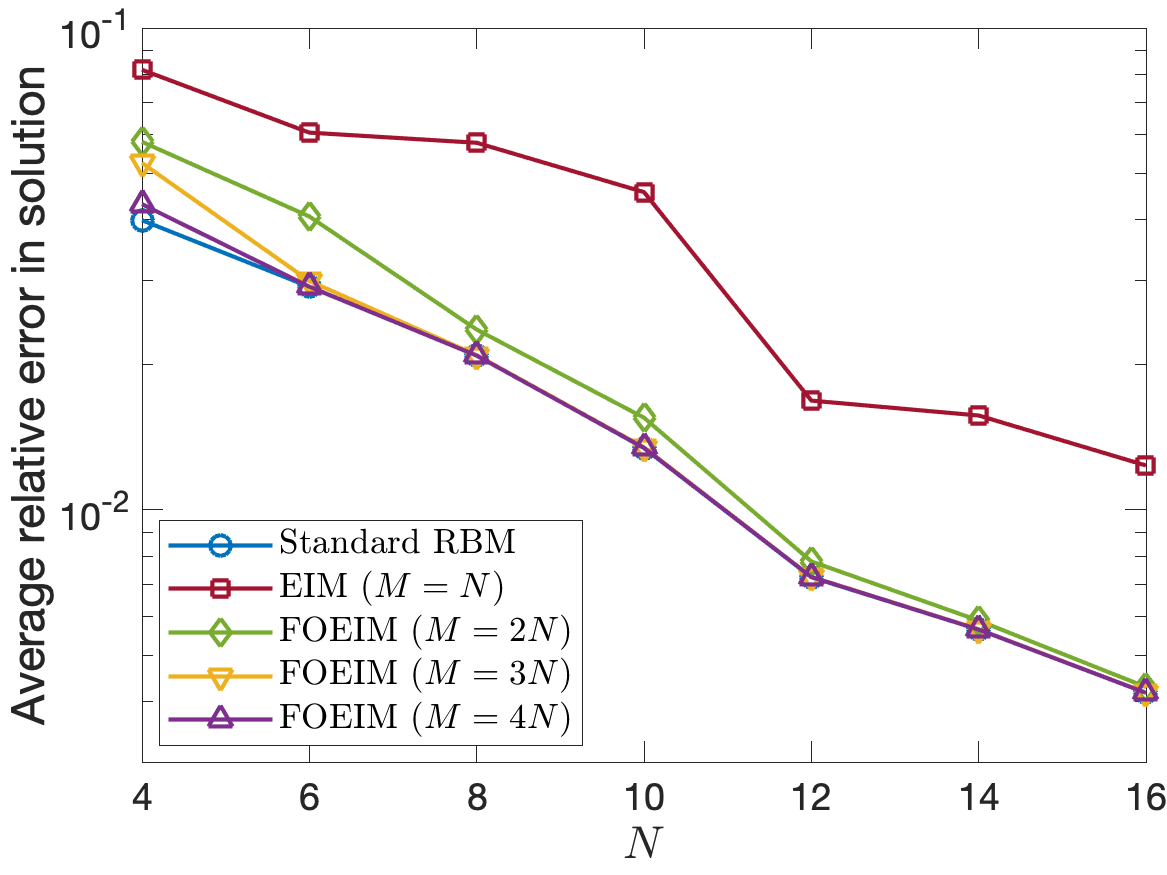}
		\caption{$\bar{\epsilon}_{N}^u$ and $\bar{\epsilon}_{N,M}^u$}
	\end{subfigure}
	\caption{Convergence of the average relative error in output (a) and solution (b) for the model problem of Section \ref{section4.1}.}
	\label{fig6}
\end{figure}

\begin{table}[htbp]
\centering
\small
	\begin{tabular}{|c|cc|cc|cc|cc|cc|}
		\cline{1-9}
    &		 
	 \multicolumn{2}{|c|}{$M=N$} & \multicolumn{2}{c|}{$M=2N$} & 
		 \multicolumn{2}{c|}{$M=3N$} &
		 \multicolumn{2}{c|}{$M=4N$} \\   
   $N$ & $\bar{\eta}_{N,M}^s$ & $\bar{\eta}_{N,M}^u$ & $\bar{\eta}_{N,M}^s$ & $\bar{\eta}_{N,M}^u$ & $\bar{\eta}_{N,M}^s$ & $\bar{\eta}_{N,M}^u$ & $\bar{\eta}_{N,M}^s$ & $\bar{\eta}_{N,M}^u$ \\
		\cline{1-9}
4  &  56.61  &  2.20  &  37.10  &  1.56  &  27.67  &  1.43  &  12.28  &  1.17  \\  
 6  &  22.17  &  2.21  &  16.38  &  1.33  &  2.46  &  1.02  &  1.21  &  1.00  \\  
 8  &  55.05  &  2.94  &  6.72  &  1.12  &  1.31  &  1.00  &  1.12  &  1.00  \\  
 10  &  346.71  &  4.53  &  63.51  &  1.11  &  6.38  &  1.01  &  3.47  &  1.00  \\  
 12  &  94.53  &  3.02  &  14.00  &  1.11  &  1.75  &  1.00  &  1.41  &  1.00  \\  
 14  &  139.88  &  3.55  &  6.67  &  1.06  &  1.57  &  1.00  &  1.06  &  1.00  \\  
 16  &  252.99  &  3.78  &  9.31  &  1.04  &  1.49  &  1.00  &  1.10  &  1.00  \\  
		\hline
	\end{tabular}
	\caption{Average effectivities for the model problem of Section \ref{section4.1}. The column with $M=N$ corresponds to the original EIM, while the columns with $M >N$ correspond  to the first-order EIM.} 
	\label{tab3}
\end{table}

We present in Table~\ref{tab4} the online
computational times to calculate $s_{N}(\bm \mu)$ and $s_{N,M}(\bm \mu)$ as a function of $N$. The values are normalized with respect to the computational time for the direct
calculation of the truth approximation output $s(\bm \mu)$. The computational saving is significant: for an relative accuracy of less than 0.001 ($N = 10$, $M = 30$) in the output, the reduction in online cost is more than a factor of 2600; this is mainly because the matrix assembly of the nonlinear terms for the truth approximation is computationally very expensive. The standard RB approximation is 2 times faster the truth FE approximation, but 1000 times slower than the RB approximation via empirical interpolation. The first-order EIM  yields as accurate approximation as the standard RB method and reduces the online computational times by  several orders of magnitude.

\begin{table}[htbp]
\centering
\small
	\begin{tabular}{|c|c|c|c|c|c|c|}
		\cline{1-7}
  & FEM  & RBM & EIM & FOEIM  & FOEIM & FOEIM \\  
   $N$ & $s(\bm \mu)$ &  $s_N(\bm \mu)$ & $M=N$ & $M=2N$ & $M=3N$ & $M=4N$ \\
		\cline{1-7}
 4 & 1 &  5.15\mbox{e-}1  &  2.08\mbox{e-}4  &  2.23\mbox{e-}4  &  2.09\mbox{e-}4  &  2.00\mbox{e-}4  \\  
 6 & 1 &  5.36\mbox{e-}1  &  2.37\mbox{e-}4  &  2.56\mbox{e-}4  &  2.45\mbox{e-}4  &  2.52\mbox{e-}4  \\  
 8 & 1 &  5.53\mbox{e-}1  &  2.55\mbox{e-}4  &  2.81\mbox{e-}4  &  2.82\mbox{e-}4  &  3.01\mbox{e-}4  \\  
 10 & 1 &  5.55\mbox{e-}1  &  2.76\mbox{e-}4  &  3.18\mbox{e-}4  &  3.25\mbox{e-}4  &  3.82\mbox{e-}4  \\  
 12 & 1  &  5.63\mbox{e-}1  &  2.99\mbox{e-}4  &  3.36\mbox{e-}4  &  3.81\mbox{e-}4  &  4.28\mbox{e-}4  \\  
 14 & 1  &  5.83\mbox{e-}1  &  3.29\mbox{e-}4  &  3.86\mbox{e-}4  &  4.44\mbox{e-}4  &  5.59\mbox{e-}4  \\  
 16 & 1 &  5.89\mbox{e-}1  &  3.46\mbox{e-}4  &  4.09\mbox{e-}4  &  4.75\mbox{e-}4  &  6.14\mbox{e-}4  \\  
		\hline
	\end{tabular}
	\caption{Online computational times (normalized with respect to the time to
  solve for $s(\bm \mu)$) for the model problem of Section \ref{section3.1}.} 
	\label{tab4}
\end{table}

\section{Conclusion}


We have presented an efficient model reduction technique for constructing accurate reduced-basis approximation of nonlinear PDEs via the first-order empirical interpolation. Although we apply our approach to elliptic problems, it can be extended to other PDEs with minor modification. Numerical results were presented to demonstrate that the first-order EIM approach provides computational savings of many orders of magnitude relative to the FE approximation and the standard RB approximation. Furthermore, the first-order EIM approach is considerably more accurate than the original EIM approach for the same dimension of the RB space $N$. Indeed, the proposed approach can be made as accurate as the standard RB approximation by increasing the number of the interpolation points.  

In this paper, we have not considered the selection of parameter sample sets and {\em a posteriori} error estimation.
The accuracy, efficiency, and reliability of a reduced-order model depend crucially on  a parameter sample set to guarantee rapid convergence, and {\em a posteriori} estimator \cite{Huynh06:CRAS_note,Sen2006} to quantify the approximation error. If {\em a posteriori} error estimates are available, they can be used to select the parameter points by using the greedy sampling method \cite{ARCME,Boyaval2009a}. Therefore, {\em a posteriori} error estimation is an important topic to be addressed in future work. \revise{We would like to point out that {\em a posteriori} error estimation procedures have been successfully developed for nonlinear  elliptic PDEs \cite{Nguyen2007} and nonlinear parabolic PDEs \cite{grepl05:_phd_thesis}.  {\em A posteriori} error estimation for nonlinear hyperbolic PDEs is an active yet challenging area of research. We would like to pursue dual-weighted residual error estimation \cite{Yano2020,Blonigan2023,Du2021} for nonlinear hyperbolic PDEs in future work.}

\revise{It is natural to extend the proposed method to higher-order derivative information such as second-order partial derivatives. We believe that the use of higher-order partial derivatives may be necessary for nonlinear hyperbolic PDEs as it may improve the accuracy and stability of ROMs compared to the use of first-order derivatives. We leave this topic for future research.}

\section*{Acknowledgements} \label{}
We would like to thank Professor Anthony T. Patera at MIT, Professor Robert M. Freund at MIT, and Professor Yvon Maday at University of Paris VI for fruitful discussions. We gratefully acknowledge a Seed Grant from the MIT Portugal Program, the United States  Department of Energy under contract DE-NA0003965 and the Air Force Office of Scientific Research under Grant No. FA9550-22-1-0356 for supporting this work.  


 \bibliographystyle{elsarticle-num} 
\bibliography{library.bib}

\begin{thebibliography}{10}
\expandafter\ifx\csname url\endcsname\relax
  \def\url#1{\texttt{#1}}\fi
\expandafter\ifx\csname urlprefix\endcsname\relax\def\urlprefix{URL }\fi
\expandafter\ifx\csname href\endcsname\relax
  \def\href#1#2{#2} \def\path#1{#1}\fi

\bibitem{rowley04:_compressible_pod}
C.~W. Rowley, T.~Colonius, R.~M. Murray, {Model reduction for compressible
  flows using POD and Galerkin projection}, Physica D. Nonlinear Phenomena
  189~(1-2) (2004) 115--129.

\bibitem{LeGresley2000}
P.~A. LeGresley, J.~J. Alonso, {Airfoil design optimization using reduced order
  models based on proper orthogonal decomposition}, in: Fluids 2000 Conference
  and Exhibit, 2000, p. 2545.
\newblock \href {https://doi.org/10.2514/6.2000-2545}
  {\path{doi:10.2514/6.2000-2545}}.

\bibitem{Knezevic2011}
D.~J. Knezevic, N.~C. Nguyen, A.~T. Patera, {Reduced basis approximation and a
  posteriori error estimation for the parametrized unsteady Boussinesq
  equations}, Mathematical Models and Methods in Applied Sciences 21~(7) (2011)
  1415--1442.
\newblock \href {https://doi.org/10.1142/S0218202511005441}
  {\path{doi:10.1142/S0218202511005441}}.

\bibitem{Lorenzi2016}
S.~Lorenzi, A.~Cammi, L.~Luzzi, G.~Rozza, {POD-Galerkin method for finite
  volume approximation of Navier–Stokes and RANS equations}, Computer Methods
  in Applied Mechanics and Engineering 311 (2016) 151--179.
\newblock \href {https://doi.org/10.1016/j.cma.2016.08.006}
  {\path{doi:10.1016/j.cma.2016.08.006}}.

\bibitem{Yano2019}
M.~Yano, {Discontinuous Galerkin reduced basis empirical quadrature procedure
  for model reduction of parametrized nonlinear conservation laws}, Advances in
  Computational Mathematics 45~(5-6) (2019) 2287--2320.
\newblock \href {https://doi.org/10.1007/s10444-019-09710-z}
  {\path{doi:10.1007/s10444-019-09710-z}}.

\bibitem{Yano2020}
M.~Yano, {Goal-oriented model reduction of parametrized nonlinear partial
  differential equations: Application to aerodynamics}, International Journal
  for Numerical Methods in Engineering 121~(23) (2020) 5200--5226.
\newblock \href {https://doi.org/10.1002/nme.6395}
  {\path{doi:10.1002/nme.6395}}.

\bibitem{Blonigan2021}
P.~J. Blonigan, F.~Rizzi, M.~Howard, J.~A. Fike, K.~T. Carlberg, {Model
  reduction for steady hypersonic aerodynamics via conservative manifold
  least-squares Petrov–Galerkin projection}, AIAA Journal 59~(4) (2021)
  1296--1312.
\newblock \href {https://doi.org/10.2514/1.J059785}
  {\path{doi:10.2514/1.J059785}}.

\bibitem{Yu2022}
J.~Yu, J.~S. Hesthaven, {Model order reduction for compressible flows solved
  using the discontinuous Galerkin methods}, Journal of Computational Physics
  468 (2022) 111452.
\newblock \href {https://doi.org/10.1016/j.jcp.2022.111452}
  {\path{doi:10.1016/j.jcp.2022.111452}}.

\bibitem{Ballarin2016}
F.~Ballarin, G.~Rozza, {POD–Galerkin monolithic reduced order models for
  parametrized fluid-structure interaction problems}, International Journal for
  Numerical Methods in Fluids 82~(12) (2016) 1010--1034.
\newblock \href {https://doi.org/10.1002/fld.4252}
  {\path{doi:10.1002/fld.4252}}.

\bibitem{Carlberg2013}
K.~Carlberg, C.~Farhat, J.~Cortial, D.~Amsallem, {The GNAT method for nonlinear
  model reduction: Effective implementation and application to computational
  fluid dynamics and turbulent flows}, Journal of Computational Physics 242
  (2013) 623--647.
\newblock \href {http://arxiv.org/abs/1207.1349} {\path{arXiv:1207.1349}},
  \href {https://doi.org/10.1016/j.jcp.2013.02.028}
  {\path{doi:10.1016/j.jcp.2013.02.028}}.

\bibitem{Du2022}
E.~Du, M.~Yano, {Efficient hyperreduction of high-order discontinuous Galerkin
  methods: Element-wise and point-wise reduced quadrature formulations},
  Journal of Computational Physics 466 (2022) 111399.
\newblock \href {https://doi.org/10.1016/j.jcp.2022.111399}
  {\path{doi:10.1016/j.jcp.2022.111399}}.

\bibitem{ARCME}
G.~Rozza, D.~B.~P. Huynh, A.~T. Patera, {Reduced basis approximation and a
  posteriori error estimation for affinely parametrized elliptic coercive
  partial differential equations: Application to transport and continuum
  mechanics}, Archives Computational Methods in Engineering 15~(4) (2008)
  229--275.

\bibitem{Huynh2007a}
D.~B. Huynh, A.~T. Patera, {Reduced basis approximation and a posteriori error
  estimation for stress intensity factors}, International Journal for Numerical
  Methods in Engineering 72~(10) (2007) 1219--1259.
\newblock \href {https://doi.org/10.1002/nme.2090}
  {\path{doi:10.1002/nme.2090}}.

\bibitem{Farhat2014}
C.~Farhat, P.~Avery, T.~Chapman, J.~Cortial, {Dimensional reduction of
  nonlinear finite element dynamic models with finite rotations and
  energy-based mesh sampling and weighting for computational efficiency},
  International Journal for Numerical Methods in Engineering 98~(9) (2014)
  625--662.
\newblock \href {https://doi.org/10.1002/nme.4668}
  {\path{doi:10.1002/nme.4668}}.

\bibitem{Farhat2015}
C.~Farhat, T.~Chapman, P.~Avery, {Structure-preserving, stability, and accuracy
  properties of the energy-conserving sampling and weighting method for the
  hyper reduction of nonlinear finite element dynamic models}, International
  Journal for Numerical Methods in Engineering 102~(5) (2015) 1077--1110.
\newblock \href {https://doi.org/10.1002/nme.4820}
  {\path{doi:10.1002/nme.4820}}.

\bibitem{Tiso2013}
P.~Tiso, D.~J. Rixen, {Discrete empirical interpolation method for finite
  element structural dynamics}, in: Conference Proceedings of the Society for
  Experimental Mechanics Series, Vol.~1, 2013, pp. 203--212.

\bibitem{Chen2010a}
Y.~Chen, J.~S. Hesthaven, Y.~Maday, J.~Rodr{\'{i}}guez,
  \href{http://epubs.siam.org/action/showAbstract?page=970{\&}volume=32{\&}issue=2{\&}journalCode=sjoce3}{{Certified
  Reduced Basis Methods and Output Bounds for the Harmonic Maxwell's
  Equations}}, SIAM Journal on Scientific Computing 32~(2) (2010) 970--996.
\newblock \href {https://doi.org/10.1137/09075250X}
  {\path{doi:10.1137/09075250X}}.

\bibitem{Vidal-Codina2018a}
F.~Vidal-Codina, N.~C. Nguyen, J.~Peraire, {Computing parametrized solutions
  for plasmonic nanogap structures}, Journal of Computational Physics 366
  (2018) 89--106.
\newblock \href {http://arxiv.org/abs/1710.06539} {\path{arXiv:1710.06539}},
  \href {https://doi.org/10.1016/j.jcp.2018.04.009}
  {\path{doi:10.1016/j.jcp.2018.04.009}}.

\bibitem{Pomplun2010}
J.~Pomplun, F.~Schmidt, {Accelerated a posteriori error estimation for the
  reduced basis method with application to 3d electromagnetic scattering
  problems}, SIAM Journal on Scientific Computing 32~(2) (2010) 498--520.
\newblock \href {https://doi.org/10.1137/090760271}
  {\path{doi:10.1137/090760271}}.

\bibitem{Manzoni2012}
A.~Manzoni, A.~Quarteroni, G.~Rozza, {Shape optimization for viscous flows by
  reduced basis methods and free-form deformation}, International Journal for
  Numerical Methods in Fluids 70~(5) (2012) 646--670.
\newblock \href {https://doi.org/10.1002/fld.2712}
  {\path{doi:10.1002/fld.2712}}.

\bibitem{Qian2017}
E.~Qian, M.~Grepl, K.~Veroy, K.~Willcox, {A Certified Trust Region Reduced
  Basis Approach to PDE-Constrained Optimization}, SIAM Journal on Scientific
  Computing 39~(5) (2017) S434--S460.
\newblock \href {https://doi.org/10.1137/16m1081981}
  {\path{doi:10.1137/16m1081981}}.

\bibitem{Hoang2013}
K.~Hoang, B.~Khoo, G.~Liu, N.~Nguyen, A.~Patera,
  \href{http://dx.doi.org/10.1080/17415977.2012.757315}{{Rapid identification
  of material properties of the interface tissue in dental implant systems
  using reduced basis method}}, Inverse Problems in Science and Engineering
  21~(8) (2013) 1310--1334.
\newblock \href {https://doi.org/10.1080/17415977.2012.757315}
  {\path{doi:10.1080/17415977.2012.757315}}.

\bibitem{Nguyen_SantaFE08}
N.~C. Nguyen, G.~Rozza, D.~B. Huynh, A.~T. Patera, {Reduced basis approximation
  and a posteriori error estimation for parametrized parabolic pdes:
  Application to real-time bayesian parameter estimation}, in: Biegler, Biros,
  Ghattas, Heinkenschloss, Keyes, Mallick, Tenorio, {van Bloemen Waanders},
  Willcox (Eds.), Large-Scale Inverse Problems and Quantification of
  Uncertainty, John Wiley and Sons, UK, 2010, pp. 151--177.
\newblock \href {https://doi.org/10.1002/9780470685853.ch8}
  {\path{doi:10.1002/9780470685853.ch8}}.

\bibitem{Galbally2010}
D.~Galbally, K.~Fidkowski, K.~Willcox, O.~Ghattas, {Non-linear model reduction
  for uncertainty quantification in large-scale inverse problems},
  International Journal for Numerical Methods in Engineering 81~(12) (2010)
  1581--1608.
\newblock \href {https://doi.org/10.1002/nme.2746}
  {\path{doi:10.1002/nme.2746}}.

\bibitem{Lieberman2010}
C.~Lieberman, K.~Willcox, O.~Ghattas, {Parameter and state model reduction for
  large-scale statistical inverse problems}, SIAM Journal on Scientific
  Computing 32~(5) (2010) 2523--2542.
\newblock \href {https://doi.org/10.1137/090775622}
  {\path{doi:10.1137/090775622}}.

\bibitem{Nguyen2008}
N.~C. Nguyen,
  \href{http://www.scopus.com/inward/record.url?eid=2-s2.0-53349103005{\&}partnerID=40{\&}md5=64852a2fac4b34c956379f980e2cf264}{{A
  multiscale reduced-basis method for parametrized elliptic partial
  differential equations with multiple scales}}, Journal of Computational
  Physics 227~(23) (2008) 9807--9822.

\bibitem{Bader2016}
E.~Bader, M.~K{\"{a}}rcher, M.~A. Grepl, K.~Veroy, {Certified reduced basis
  methods for parametrized distributed elliptic optimal control problems with
  control constraints}, SIAM Journal on Scientific Computing 38~(6) (2016)
  A3921--A3946.
\newblock \href {https://doi.org/10.1137/16M1059898}
  {\path{doi:10.1137/16M1059898}}.

\bibitem{Karcher2018}
M.~K{\"{a}}rcher, Z.~Tokoutsi, M.~A. Grepl, K.~Veroy, {Certified Reduced Basis
  Methods for Parametrized Elliptic Optimal Control Problems with Distributed
  Controls}, Journal of Scientific Computing 75~(1) (2018) 276--307.
\newblock \href {https://doi.org/10.1007/s10915-017-0539-z}
  {\path{doi:10.1007/s10915-017-0539-z}}.

\bibitem{Karcher2018a}
M.~K{\"{a}}rcher, S.~Boyaval, M.~A. Grepl, K.~Veroy, {Reduced basis
  approximation and a posteriori error bounds for 4D-Var data assimilation},
  Optimization and Engineering 19~(3) (2018) 663--695.
\newblock \href {http://arxiv.org/abs/1802.02328} {\path{arXiv:1802.02328}},
  \href {https://doi.org/10.1007/s11081-018-9389-2}
  {\path{doi:10.1007/s11081-018-9389-2}}.

\bibitem{Maday2013}
Y.~Maday, O.~Mula, {A generalized empirical interpolation method: Application
  of reduced basis techniques to data assimilation}, in: Springer INdAM Series,
  Vol.~4, Springer, 2013, pp. 221--235.
\newblock \href {http://arxiv.org/abs/1512.00683} {\path{arXiv:1512.00683}},
  \href {https://doi.org/10.1007/978-88-470-2592-9-13}
  {\path{doi:10.1007/978-88-470-2592-9-13}}.

\bibitem{Maday2015b}
Y.~Maday, A.~T. Patera, J.~D. Penn, M.~Yano, {A parameterized-background
  data-weak approach to variational data assimilation: Formulation, analysis,
  and application to acoustics}, International Journal for Numerical Methods in
  Engineering 102~(5) (2015) 933--965.
\newblock \href {https://doi.org/10.1002/nme.4747}
  {\path{doi:10.1002/nme.4747}}.

\bibitem{prudhomme02:_reliab_real_time_solut_param}
C.~Prud'homme, D.~Rovas, K.~Veroy, Y.~Maday, A.~T. Patera, G.~Turinici,
  {Reliable real-time solution of parametrized partial differential equations:
  Reduced-basis output bounds methods}, Journal of Fluids Engineering 124~(1)
  (2002) 70--80.

\bibitem{Nguyen2009b}
N.~C. Nguyen, G.~Rozza, A.~T. Patera, {Reduced basis approximation and a
  posteriori error estimation for the time-dependent viscous Burgers'
  equation}, Calcolo 46~(3) (2009) 157--185.
\newblock \href {https://doi.org/10.1007/s10092-009-0005-x}
  {\path{doi:10.1007/s10092-009-0005-x}}.

\bibitem{Grepl2007a}
M.~A. Grepl, Y.~Maday, N.~C. Nguyen, A.~T. Patera,
  \href{http://www.esaim-m2an.org/10.1051/m2an:2007031}{{Efficient
  reduced-basis treatment of nonaffine and nonlinear partial differential
  equations}}, Mathematical Modelling and Numerical Analysis 41~(3) (2007)
  575--605.
\newblock \href {https://doi.org/10.1051/m2an:2007031}
  {\path{doi:10.1051/m2an:2007031}}.

\bibitem{Nguyen2008d}
N.~C. Nguyen, J.~Peraire, \href{http://doi.wiley.com/10.1002/nme.2309}{{An
  efficient reduced-order modeling approach for non-linear parametrized partial
  differential equations}}, International Journal for Numerical Methods in
  Engineering 76~(1) (2008) 27--55.
\newblock \href {https://doi.org/10.1002/nme.2309}
  {\path{doi:10.1002/nme.2309}}.

\bibitem{weile01:_reduc_krylov}
D.~S. Weile, E.~Michielssen, K.~Gallivan, {Reduced-order modeling of
  multiscreen frequency-selective surfaces using Krylov-based rational
  interpolation}, IEEE Transactions on Antennas and Propagation 49~(5) (2001)
  801--813.
\newblock \href {https://doi.org/10.1109/8.929635}
  {\path{doi:10.1109/8.929635}}.

\bibitem{phillips03:_weakly_nonlinearMOR}
J.~R. Phillips, {Projection-based approaches for model reduction of weakly
  nonlinear, time-varying systems}, IEEE Transactions on Computer-Aided Design
  of Integrated Circuits and Systems 22~(2) (2003) 171--187.
\newblock \href {https://doi.org/10.1109/TCAD.2002.806605}
  {\path{doi:10.1109/TCAD.2002.806605}}.

\bibitem{rewienski03:_trajectory_approach}
M.~Rewie{\'{n}}ski, J.~White, {A trajectory piecewise-linear approach to model
  order reduction and fast simulation of nonlinear circuits and micromachined
  devices}, IEEE Transactions on Computer-Aided Design of Integrated Circuits
  and Systems 22~(2) (2003) 155--170.
\newblock \href {https://doi.org/10.1109/TCAD.2002.806601}
  {\path{doi:10.1109/TCAD.2002.806601}}.

\bibitem{Barrault2004a}
M.~Barrault, Y.~Maday, N.~C. Nguyen, A.~T. Patera,
  \href{http://linkinghub.elsevier.com/retrieve/pii/S1631073X04004248}{{An
  ‘empirical interpolation' method: application to efficient reduced-basis
  discretization of partial differential equations}}, Comptes Rendus
  Mathematique 339~(9) (2004) 667--672.
\newblock \href {https://doi.org/10.1016/j.crma.2004.08.006}
  {\path{doi:10.1016/j.crma.2004.08.006}}.

\bibitem{Nguyen2007}
N.~C. Nguyen,
  \href{http://linkinghub.elsevier.com/retrieve/pii/S0021999107003749}{{A
  posteriori error estimation and basis adaptivity for reduced-basis
  approximation of nonaffine-parametrized linear elliptic partial differential
  equations}}, Journal of Computational Physics 227~(2) (2007) 983--1006.
\newblock \href {https://doi.org/10.1016/j.jcp.2007.08.031}
  {\path{doi:10.1016/j.jcp.2007.08.031}}.

\bibitem{Drohmann2012}
M.~Drohmann, B.~Haasdonk, M.~Ohlberger, {Reduced basis approximation for
  nonlinear parametrized evolution equations based on empirical operator
  interpolation}, SIAM Journal on Scientific Computing 34~(2) (2012)
  A937--A969.
\newblock \href {https://doi.org/10.1137/10081157X}
  {\path{doi:10.1137/10081157X}}.

\bibitem{Hesthaven2014}
J.~S. Hesthaven, B.~Stamm, S.~Zhang, {Efficient greedy algorithms for
  high-dimensional parameter spaces with applications to empirical
  interpolation and reduced basis methods}, ESAIM: Mathematical Modelling and
  Numerical Analysis 48~(1) (2014) 259--283.
\newblock \href {https://doi.org/10.1051/m2an/2013100}
  {\path{doi:10.1051/m2an/2013100}}.

\bibitem{Hesthaven2022}
J.~S. Hesthaven, C.~Pagliantini, G.~Rozza, {Reduced basis methods for
  time-dependent problems}, Acta Numerica 31 (2022) 265--345.
\newblock \href {https://doi.org/10.1017/S0962492922000058}
  {\path{doi:10.1017/S0962492922000058}}.

\bibitem{Chen2021}
Y.~Chen, S.~Gottlieb, L.~Ji, Y.~Maday, {An EIM-degradation free reduced basis
  method via over collocation and residual hyper reduction-based error
  estimation}, Journal of Computational Physics 444 (2021) 110545.
\newblock \href {https://doi.org/10.1016/j.jcp.2021.110545}
  {\path{doi:10.1016/j.jcp.2021.110545}}.

\bibitem{Nguyen2008a}
N.~C. Nguyen, A.~T. Patera, J.~Peraire,
  \href{http://doi.wiley.com/10.1002/nme.2086}{{A 'best points' interpolation
  method for efficient approximation of parametrized functions}}, International
  Journal for Numerical Methods in Engineering 73~(4) (2008) 521--543.
\newblock \href {https://doi.org/10.1002/nme.2086}
  {\path{doi:10.1002/nme.2086}}.

\bibitem{Kramer2019}
B.~Kramer, K.~E. Willcox, {Nonlinear model order reduction via lifting
  transformations and proper orthogonal decomposition}, AIAA Journal 57~(6)
  (2019) 2297--2307.
\newblock \href {http://arxiv.org/abs/1808.02086} {\path{arXiv:1808.02086}},
  \href {https://doi.org/10.2514/1.J057791} {\path{doi:10.2514/1.J057791}}.

\bibitem{Yano2019a}
M.~Yano, A.~T. Patera, {An LP empirical quadrature procedure for reduced basis
  treatment of parametrized nonlinear PDEs}, Computer Methods in Applied
  Mechanics and Engineering 344 (2019) 1104--1123.
\newblock \href {https://doi.org/10.1016/j.cma.2018.02.028}
  {\path{doi:10.1016/j.cma.2018.02.028}}.

\bibitem{Chaturantabut2010}
S.~Chaturantabut, D.~C. Sorensen, {Nonlinear model reduction via discrete
  empirical interpolation}, SIAM Journal on Scientific Computing 32~(5) (2010)
  2737--2764.
\newblock \href {https://doi.org/10.1137/090766498}
  {\path{doi:10.1137/090766498}}.

\bibitem{Maierhofer2022}
J.~Maierhofer, D.~J. Rixen, {Model order reduction using hyperreduction methods
  (DEIM, ECSW) for magnetodynamic FEM problems}, Finite Elements in Analysis
  and Design 209 (2022) 103793.
\newblock \href {https://doi.org/10.1016/j.finel.2022.103793}
  {\path{doi:10.1016/j.finel.2022.103793}}.

\bibitem{Astrid2008}
P.~Astrid, S.~Weiland, K.~Willcox, T.~Backx, {Missing point estimation in
  models described by proper orthogonal decomposition}, IEEE Transactions on
  Automatic Control 53~(10) (2008) 2237--2251.
\newblock \href {https://doi.org/10.1109/TAC.2008.2006102}
  {\path{doi:10.1109/TAC.2008.2006102}}.

\bibitem{Peherstorfer2020}
B.~Peherstorfer, Z.~Drmac, S.~Gugercin, {Stability of discrete empirical
  interpolation and gappy proper orthogonal decomposition with randomized and
  deterministic sampling points}, SIAM Journal on Scientific Computing 42~(5)
  (2020) A2837--A2864.
\newblock \href {http://arxiv.org/abs/1808.10473} {\path{arXiv:1808.10473}},
  \href {https://doi.org/10.1137/19M1307391} {\path{doi:10.1137/19M1307391}}.

\bibitem{Eftang2012b}
J.~L. Eftang, B.~Stamm, {Parameter multi-domain 'hp' empirical interpolation},
  International Journal for Numerical Methods in Engineering 90~(4) (2012)
  412--428.
\newblock \href {https://doi.org/10.1002/nme.3327}
  {\path{doi:10.1002/nme.3327}}.

\bibitem{Peherstorfer2014}
B.~Peherstorfer, D.~Butnaru, K.~Willcox, H.~J. Bungartz, {Localized discrete
  empirical interpolation method}, SIAM Journal on Scientific Computing 36~(1)
  (2014).
\newblock \href {https://doi.org/10.1137/130924408}
  {\path{doi:10.1137/130924408}}.

\bibitem{everson95karhunenloeve}
R.~Everson, L.~Sirovich, {Karhunen-Loeve procedure for gappy data}, Opt. Soc.
  Am. A 12~(8) (1995) 1657--1664.

\bibitem{willcox06:_gappy}
K.~Willcox, {Unsteady Flow Sensing and Estimation via the Gappy Proper
  Orthogonal Decomposition}, Computers and Fluids 35 (2006) 208--226.

\bibitem{Zimmermann2016}
R.~Zimmermann, K.~Willcox, {An accelerated greedy missing point estimation
  procedure}, SIAM Journal on Scientific Computing 38~(5) (2016) A2827--A2850.
\newblock \href {https://doi.org/10.1137/15M1042899}
  {\path{doi:10.1137/15M1042899}}.

\bibitem{Argaud2017}
J.~P. Argaud, B.~Bouriquet, H.~Gong, Y.~Maday, O.~Mula, {Stabilization of
  (G)EIM in Presence of Measurement Noise: Application to Nuclear Reactor
  Physics}, in: Lecture Notes in Computational Science and Engineering, Vol.
  119, 2017, pp. 133--145.
\newblock \href {http://arxiv.org/abs/1611.02219} {\path{arXiv:1611.02219}}.

\bibitem{Guillaume97}
P.~Guillaume, M.~Masmoudi, {Solution to the time-harmonic Maxwell's equations
  in a waveguide: use of higher-order derivatives for solving the discrete
  problem}, SIAM J. Numer. Anal. 34~(4) (1997) 1306--1330.

\bibitem{ito98:_reduc_basis_method_contr_probl_gover_pdes}
K.~Ito, S.~S. Ravindran, {A Reduced Basis Method for Control Problems Governed
  by {\{}PDEs{\}}}, in: W.~Desch, F.~Kappel, K.~Kunisch (Eds.), Control and
  Estimation of Distributed Parameter Systems, Birkh{\"{a}}user, 1998, pp.
  153--168.

\bibitem{Maday2008}
Y.~Maday, N.~C. Nguyen, A.~T. Patera, G.~S. Pau,
  \href{http://www.aimsciences.org/journals/displayArticles.jsp?paperID=3753}{{A
  general multipurpose interpolation procedure: The magic points}},
  Communications on Pure and Applied Analysis 8~(1) (2009) 383--404.
\newblock \href {https://doi.org/10.3934/cpaa.2009.8.383}
  {\path{doi:10.3934/cpaa.2009.8.383}}.

\bibitem{Vila-Perez2022}
J.~Vila-P{\'{e}}rez, R.~L. {Van Heyningen}, N.-C. Nguyen, J.~Peraire,
  \href{https://www.sciencedirect.com/science/article/pii/S2352711022001303}{{Exasim:
  Generating discontinuous Galerkin codes for numerical solutions of partial
  differential equations on graphics processors}}, SoftwareX 20 (2022) 101212.
\newblock \href {https://doi.org/https://doi.org/10.1016/j.softx.2022.101212}
  {\path{doi:https://doi.org/10.1016/j.softx.2022.101212}}.

\bibitem{Huynh06:CRAS_note}
D.~B.~P. Huynh, G.~Rozza, S.~Sen, A.~T. Patera, {A successive constraint linear
  optimization method for lower bounds of parametric coercivity and inf-sup
  stability constants}, C. R. Acad. Sci. Paris, Analyse Num{\'{e}}rique 345~(8)
  (2007) 473--478.

\bibitem{Sen2006}
S.~Sen, K.~Veroy, D.~B.~P. Huynh, S.~Deparis, N.~C. Nguyen, A.~T. Patera,
  \href{http://linkinghub.elsevier.com/retrieve/pii/S0021999106000830}{{Natural
  norm'' a posteriori error estimators for reduced basis approximations}},
  Journal of Computational Physics 217~(1) (2006) 37--62.
\newblock \href {https://doi.org/10.1016/j.jcp.2006.02.012}
  {\path{doi:10.1016/j.jcp.2006.02.012}}.

\bibitem{Boyaval2009a}
S.~Boyaval, C.~L. Bris, Y.~Maday, N.~C. Nguyen, A.~T. Patera,
  \href{http://linkinghub.elsevier.com/retrieve/pii/S0045782509002114}{{A
  reduced basis approach for variational problems with stochastic parameters:
  Application to heat conduction with variable Robin coefficient}}, Computer
  Methods in Applied Mechanics and Engineering 198~(41-44) (2009) 3187--3206.
\newblock \href {https://doi.org/10.1016/j.cma.2009.05.019}
  {\path{doi:10.1016/j.cma.2009.05.019}}.

\bibitem{grepl05:_phd_thesis}
M.~Grepl, {Reduced-Basis Approximations and {\{}$\backslash$em A
  Posteriori$\backslash$/{\}} Error Estimation for Parabolic Partial
  Differential Equations}, Ph.D. thesis, Massachusetts Institute of Technology,
  Cambridge, MA (may 2005).

\bibitem{Blonigan2023}
P.~J. Blonigan, E.~J. Parish, {Evaluation of dual-weighted residual and machine
  learning error estimation for projection-based reduced-order models of steady
  partial differential equations}, Computer Methods in Applied Mechanics and
  Engineering 409 (2023) 115988.
\newblock \href {https://doi.org/10.1016/j.cma.2023.115988}
  {\path{doi:10.1016/j.cma.2023.115988}}.

\bibitem{Du2021}
E.~Du, M.~Sleeman, {Adaptive Discontinuous-Galerkin Reduced-Basis
  Reduced-Quadrature Method for Many-Query CFD Problems}, in: AIAA Aviation and
  Aeronautics Forum and Exposition, AIAA AVIATION Forum 2021, 2021, pp.
  AIAA--2021--2716.
\newblock \href {https://doi.org/10.2514/6.2021-2716}
  {\path{doi:10.2514/6.2021-2716}}.

\end{thebibliography}




\end{document}